\documentclass[english]{article}
\usepackage[latin9]{inputenc}
\usepackage{amsmath}
\usepackage{amssymb}
\usepackage{esint}

\makeatletter
\title{Newtheorem and theoremstyle test}
\author{Michael Downes\\updated by Barbara Beeton}
\usepackage[exercise]{amsthm}
\newtheorem{thm}{Theorem}[section]
\newtheorem{cor}[thm]{Corollary}

\newtheorem{lem}[thm]{Lemma}

\theoremstyle{remark}

\theoremstyle{plain}
\newtheorem{Def}{Definition}
\newtheoremstyle{note}
  {3pt}
  {3pt}
  {}
  {}
  {\itshape}
  {:}
  {.5em}
  {}
\theoremstyle{note}

\newtheoremstyle{citing}
  {3pt}
  {3pt}
  {\itshape}
  {}
  {\bfseries}
  {.}
  {.5em}
  {\thmnote{#3}}
\theoremstyle{citing}
\newtheoremstyle{break}
  {9pt}
  {9pt}
  {\itshape}
  {}
  {\bfseries}
  {.}
  {\newline}
  {}
\theoremstyle{break}

\theoremstyle{exercise}

\swapnumbers
\theoremstyle{plain}

\let\lvert=|\let\rvert=|

\newcommand{\finelim}{\operatornamewithlimits{finelim}}
\usepackage{babel}
\usepackage{babel}
\usepackage{babel}
\usepackage{babel}
\DeclareMathOperator*{\osc}{osc}

\makeatother

\usepackage{babel}
\begin{document}
\title{Cohn-Vossen theory for locally conformally flat manifolds}
\author{Shiguang Ma}
\maketitle
\begin{abstract}
We establish a refined singularity estimate for nonnegative $n$-superharmonic
functions. For complete noncompact locally conformally flat manifolds
with nonnegative Ricci curvature, we analytically characterize the
volume growth, verify Yau\textquoteright s conjecture on the Cohn\textendash Vossen
inequality, and prove a sharp gap theorem that removes all auxiliary
assumptions from earlier works.
\end{abstract}

\section{Introduction}

The Cohn-Vossen inequality stands as a foundational result in global
differential geometry, forging a profound link between the total curvature
integral and topological invariants of complete non-compact 2-dimensional
Riemannian surfaces. For a complete, non-compact connected surface
$(M^{2},g)$ with nonnegative Gaussian curvature $K\ge0$, this inequality
asserts that
\begin{equation}
\int_{M}Kd\mu_{g}\le2\pi\chi(M),\label{classical Cohn-Vossen inequality}
\end{equation}
where $\chi(M)$ denotes the Euler characteristic of $M$. Cohn-Vossen
\cite{CV1935} proved this inequality for surfaces $M$ with analytic
metrics. Later, Huber \cite{Huber1957} extended this inequality to
metrics with weaker regularity. Finn \cite{Finn1965} characterized
the deficit of (\ref{classical Cohn-Vossen inequality}) as 
\[
\chi(M)-\frac{1}{2\pi}\int_{M}Kd\mu_{g}=\sum\nu_{j},
\]
with $\nu_{j}$ being the isoperimetric ratio of the $j$th ends.

Extending this classical $2$-dimensional result to higher-dimensional
manifolds, and establishing a precise correspondence between curvature
integral asymptotics, topological invariants and asymptotic geometric
quantities, has long been a central open problem in global differential
geometry.

In this paper, $B^{g}_{l}(p)$ is the geodesic ball centered at $p$
with radius $l$, with respect to the metric $g$ and $B_{r}(p)$
or $B(p,r)$ is a geodesic ball of $\mathbb{R}^{n}$. ${\rm meas}(\Omega)$
represents the Lebesgue measure of $\Omega\subset\mathbb{R}^{n}$.
$|\cdot|$ denotes the Euclidean norm on $\mathbb{R}^{n}$, the order
of a group, or the $n-1$ dimensional Hausdorff measure of a hypersurface
in $\mathbb{R}^{n}$; no ambiguity will arise. $\mu_{g}$ denotes
the volume form of a metric $g$. $V_{g}(\Omega)$ is the volume of
$\Omega$ with respect to the metric $g$. 

In 1992, Yau\cite{Yau1992} conjectured that for a complete non-compact
$n$-dimensional manifold $(M^{n},g)$ with $Ric_{g}\ge0$, the integral
of the scalar curvature $R_{g}$ over geodesic balls $B^{g}_{l}(p)$
satisfies
\[
\limsup_{l\to\infty}\frac{\int_{B^{g}_{l}(p)}R_{g}d\mu_{g}}{l^{n-2}}<\infty.
\]
He further generalized this conjecture to the elementary symmetric
polynomials $\sigma_{k},(k=1,\cdots,n)$ of the Ricci tensor, seeking
a unified high-dimensional analog of the Cohn-Vossen inequality that
connects the asymptotic behavior of curvature tensors to the global
geometric properties of manifolds. However, Yang\cite{Yang2013} constructed
explicit counterexamples showing that Yau\textquoteright s conjecture
fails for $\sigma_{k}$ with $k\ge2$. 

Subsequent research has made progress on Yau\textquoteright s conjecture
for the scalar curvature case in various specialized geometric settings.
Xu\cite{Xu2020,Xu2024} and Zhu \cite{Zh2022} focused on $3$-dimensional
manifolds with a pole. In particular, \cite{Xu2024} showed that,
if $(M^{3},g)$ is a manifold with a pole and $Ric_{g}\ge0,$ then
\[
\lim_{l\to\infty}\frac{1}{l}\int_{B^{g}_{l}(p)}R_{g}=8\pi(1-\nu).
\]
In this paper $\nu$ is always defined to be the asymptotic volume
ratio of $(M^{n},g)$, i.e. 
\[
\nu:=\lim_{l\to\infty}\frac{V_{g}(B^{g}_{l}(p))}{\omega_{n}l^{n}}\in[0,1].
\]
 For 3-dimensional general manifolds, Munteanu and Wang\cite{MuWa2025}
proved that, for complete non-compact 3-manifolds $(M,g)$, if $Ric_{g}\ge0$
and $0<C_{1}\le R_{g}\le C_{2}$, then 
\[
\limsup_{l\to\infty}\frac{1}{l}\int_{B^{g}_{l}(p)}R_{g}\le8\pi.
\]
Liu \cite{Liu2022} focused on Kähler manifolds with $Ric\ge0$. If
it has Euclidean volume growth and satisfies either (I) $M$ has nonnegative
bisectional curvature or (II) the asymptotic volume ratio is no smaller
than $1-\varepsilon(n)$, where $\varepsilon(n)$ is a small number
depending only on $n$. Then $l^{-2n+2k}\int_{B^{g}_{l}(p)}Ric^{k}\wedge\omega^{n-k}$
is bounded as $l\to\infty$. Very recently, \cite{De2026} has also
obtained several results of the Cohn-Vossen type. 

Since Cohn-Vossen inequality in dimension 2 is conformal in nature,
extending the Cohn-Vossen inequality to locally conformally flat manifolds
in higher dimensions is an important direction. For locally conformally
flat manifolds of dimension $4$, \cite{CQY2000A,CQY2000B} proved
the corresponding inequality for $Q_{4}$ curvature and characterized
the deficit also as certain isoperimetric ratio of the ends. Very
recently, \cite{LWX2025} extended Cohn-Vossen inequality to $Q_{2k},1\le k\le\frac{n}{2}$
curvature on $\mathbb{R}^{n}$ for even $n\ge4$. In this paper, we
study Cohn-Vossen inequality for scalar curvature. 

The paper also studies a gap theorem, for which, the readers may refer
to \cite{MSY1981,GW1982,GPZ1994,CZ2002,Ni2012,Ma2016,HL2025,LWX2025}
and the reference therein. In particular, \cite{Ni2012} proved that
for complete noncompact Kähler manifold $(M,g)$ of complex dimension
$m$ and with nonnegative holomorphic bisectional curvature, if 
\[
\frac{1}{V_{g}(B^{g}_{l}(p))}\int_{B^{g}_{l}(p)}R_{g}d\mu_{g}=o(\frac{1}{l^{2}}),
\]
 then $(M,g)$ is flat. In the setting of locally conformally flat
manifolds, \cite[Main Theorem]{CZ2002} showed that, if $Ric\ge0$,
$R_{g}\le C$ and 
\begin{equation}
\frac{1}{V_{g}(B^{g}_{l}(0))}\int_{B^{g}_{l}(0)}R_{g}=o(\frac{1}{l^{2}}),\label{Scalar curvature decay assumption 1}
\end{equation}
 then $g$ is flat. Other gap theorems in the setting of locally conformally
flat manifolds with $Ric\ge0$ use
\begin{equation}
\int^{\infty}_{0}\frac{l}{V_{g}(B^{g}_{l}(0))}\left(\int_{B^{g}_{l}(0)}R_{g}\right)dl<+\infty,\label{Scalar curvature decay assumption 2}
\end{equation}
instead of (\ref{Scalar curvature decay assumption 1}). Under condition
(\ref{Scalar curvature decay assumption 2}), $Ric\le C$ and non-parabolicity,
\cite{Ma2016} proved a gap theorem. And recently, under $n\ge5$
and 
\[
\int^{\infty}_{0}\frac{l}{V_{g}(B^{g}_{l}(0))}\left(\int_{B^{g}_{l}(0)}R_{g}\right)dl<\varepsilon_{0}(n),
\]
\cite{HL2025} proved a gap theorem. It is interesting to remove the
auxiliary assumptions from \cite{CZ2002}, \cite{Ma2016} and \cite{HL2025},
and prove a clean theorem like \cite{Ni2012}. 

Our main results are obtained by converting geometric problems on
locally conformally flat manifolds into singularity problems for the
$n$-Laplace equation via conformal change of metrics and inversion.

Recall that for a conformal metric on $\mathbb{R}^{n}$, $\bar{g}=e^{2w}g_{edu}$,
as is shown in \cite{MQ2021,MQ2022}, the conformal factor $w$ satisfies
the key quasilinear equation 
\[
-\Delta_{n}w=|\nabla w|^{n-2}Ric_{\bar{g}}\left(\frac{\bar{\nabla}w}{|\bar{\nabla}w|_{\bar{g}}},\frac{\bar{\nabla}w}{|\bar{\nabla}w|_{\bar{g}}}\right)e^{2w},
\]
where $\Delta_{n}w={\rm div}(|\nabla w|^{n-2}\nabla w)$. If we do
an inversion $u(x)=w(x/|x|^{2})-2\log|x|$, we can consider the asymptotic
behavior of $u$ at the origin. This places $n$-superharmonic functions
and their singular behavior at the core of our approach.

Our first main result is proved in Section 2. We establish the refined
singularity estimates for $n$-superharmonic functions and introduce
the class of strong $\mathcal{E}$-sets, which provide tight control
over the exceptional sets. 

For a function $f(x):B_{\delta}(0)\to\mathbb{R},$ we define 
\[
\underline{f}(r)=\inf\{f(x);|x|=r\}.
\]

\begin{thm}\label{main thm1 refined analytic result of u}

If $u\ge0$ is $n$-superharmonic in $B_{1}(0)$, then 

\[
u(x)=\underline{u}(|x|)+o(1),x\to0,x\notin E,
\]
for some strong $\mathcal{E}$-set $E$. 

\end{thm}

Theorem \ref{main thm1 refined analytic result of u} is a generalization
of \cite[Theorem 2]{Ha1960}. See also \cite{HK1976}. This theorem
is the engine of the entire paper. 

In Section 3, we apply Theorem \ref{main thm1 refined analytic result of u}
to locally conformally flat manifolds and derive a sharp, explicit
relation between the asymptotic exponent of the conformal factor and
the asymptotic volume ratio. Our second main result is as follows.

\begin{thm}\label{main thm2 volume ratio}

Suppose that $(\mathbb{R}^{n},g=e^{2w}g_{edu})$ is a complete manifold
with $Ric_{g}\ge0$. Then 
\[
\nu=(1-\tilde{m})^{n-1},
\]
 where 
\[
\tilde{m}=-\liminf_{x\to\infty}\frac{w(x)}{\log|x|}\in[0,1].
\]

\end{thm}

As the third main result, in Section 4, we obtain a complete resolution
of Yau\textquoteright s conjecture for locally conformally flat manifolds,
with explicit limiting formulas that depend only on the dimension
and the asymptotic volume ratio of the manifold.

\begin{thm}\label{main thm 3}

For $n\ge3$, suppose $(M^{n},g)$ is a complete non-compact locally
conformally flat manifold with nonnegative Ricci curvature. Let $\nu$
denote its asymptotic volume ratio. Then for any fixed point $p\in M$,
the limit 
\[
\lim_{l\to\infty}l^{2-n}\int_{B^{g}_{l}(p)}R_{g}d\mu_{g}
\]
exists and is given by an explicit expression:
\begin{itemize}
\item If $n>3$, the limit is 
\[
(n-1)(\nu^{\frac{n-3}{n-1}}-\nu)|\mathbb{S}^{n-1}|;
\]
\item If $n=3$, the limit depends on the universal cover:
\begin{itemize}
\item If the universal cover is not $\mathbb{S}^{2}_{a}\times\mathbb{R}$
or flat $\mathbb{R}^{3}$, then the limit is $8\pi(1-\nu);$
\item If the universal cover is $\mathbb{S}^{2}_{a}\times\mathbb{R}$, then
$\pi_{1}(M)$ is finite and the limit equals $16\pi/|\pi_{1}(M)|;$
\item If $M$ is flat, then the limit is $0$. 
\end{itemize}
\end{itemize}
\end{thm}

In particular, Theorem \ref{main thm 3} confirms Yau\textquoteright s
conjecture for all locally conformally flat manifolds with nonnegative
Ricci curvature.

As the fourth main result, in Section 4, we prove a sharp gap theorem
that removes the auxiliary assumptions of \cite{CZ2002}, \cite{Ma2016}
and \cite{HL2025}.

\begin{thm}\label{main thm 4 generalization of Chen-Zhu's theorem}

Let $n\ge3$, and let $(M^{n},g)$ be a complete noncompact locally
conformally flat manifold with nonnegative Ricci curvature. If (\ref{Scalar curvature decay assumption 1})
or (\ref{Scalar curvature decay assumption 2}) holds, $(M,g)$ is
flat. 

\end{thm}

These results provide the first complete high-dimensional Cohn\textendash Vossen
theory in the locally conformally flat setting.

The mechanism developed in this paper appears to extend beyond the
specific setting of the $n$-Laplace equation. A fundamental feature
underlying the argument is the emergence of logarithmic singularities
at the critical conformal scaling order. In this regime, geometric
quantities at infinity become encoded by logarithmic masses after
inversion, while asymptotic rigidity can be recovered through refined
control outside exceptional sets. This suggests a broader correspondence
between critical conformally invariant curvature equations and logarithmic
singularity rigidity. For example, similar phenomena are expected
to arise for four-dimensional $Q$-curvature equations associated
with the Paneitz operator, whose Green kernel also exhibits logarithmic
asymptotics.

This work initiates an extension of higher-dimensional Cohn\textendash Vossen
theory beyond locally conformally flat manifolds. It indicates that
quasiconformal geometry may serve as a viable framework for broader
cases, and we expect to apply quasiconformal mappings to explore the
asymptotic geometry of relevant manifolds, though the existence of
well-behaved such mappings remains a key challenge.

\section{Refined singularity estimates for $n$-superharmonic functions}

In this section, we establish a refined singularity estimate for nonnegative
$n$-superharmonic functions, which significantly improves Theorem
1.1 in \cite{MQ2021}. 

\subsection{Capacity, thinness and strong $\mathcal{E}$-set }

\begin{Def}

For a compact subset $K$ of a domain $\Omega$ in Euclidean space
$\mathbb{R}^{n}$, we define the $n$-capacity as 
\[
{\rm cap}_{n}(K,\Omega)=\inf\left\{ \int_{\Omega}|\nabla u|^{n}dx;u\in C^{\infty}_{0}(\Omega),u\ge1\,{\rm on}\,K\right\} .
\]
 For an arbitrary subset $E$ of $\Omega$, the $n$-capacity is defined
as 
\[
{\rm cap}_{n}(E,\Omega)=\inf_{{\rm open}G\subset\Omega\,{\rm that\,contains\,}E}\sup_{{\rm compact}\,K\subset G}{\rm cap}_{n}(K,\Omega).
\]

\[
c_{1,n}(K)=\inf\{\|u\|^{n}_{W^{1,n}}:u\in C^{\infty}_{0}(\mathbb{R}^{n}),u\ge1\,{\rm on}\,{\rm compact}\,{\rm set}\,K\},
\]
and the definition is extended to arbitrary sets similarly.

\end{Def}

In the following, we use $\mathring{D}$ to represent the interior
of a set $D$, which is an open subset of $D$. 

\begin{lem}\label{comparision of two capacities}

There is a uniform $C>0$ such that, $\forall E\subset B_{1}(0)$,
\[
C^{-1}{\rm cap}_{n}(E,B_{2}(0))\le c_{1,n}(E)\le C{\rm cap}_{n}(E,B_{2}(0)).
\]

\end{lem}

\begin{proof}

We only prove for a compact subset $E$. We fixed some $\phi\in C^{\infty}_{0}(B_{2}(0))$
and $\phi\equiv1$ in $B_{1}(0)$ with $|\nabla\phi|\le C$. 

For any $u\in C^{\infty}_{0}(\mathbb{R}^{n})$, $u\ge1$ on $E$,
it is easy to check
\[
\text{cap}_{n}(E,B_{2}(0))\le\int_{B_{2}(0)}|\nabla\phi u|^{n}dx\le C\|u\|^{n}_{W^{1,n}(\mathbb{R}^{n})}.
\]
So $\text{cap}_{n}(E,B_{2}(0))\le Cc_{1,n}(E)$. 

On the other hand for $u\in C^{\infty}_{0}(B_{2}(0))$, $u\ge1$ on
$E$, by Poincaré inequality 
\[
\|u\|^{n}_{L^{n}(B_{2})}\le C\|\nabla u\|^{n}_{L^{n}(B_{2})}.
\]
So $c_{1,n}(E)\le C\text{cap}_{n}(E,B_{2}(0))$. 

\end{proof}

We now introduce the notions of thinness that we use. Let
\[
A_{\alpha,\beta}(x_{0})=\{x\in\mathbb{R}^{n};\alpha\le|x-x_{0}|\le\beta\},
\]
 and $A_{\alpha,\beta}=A_{\alpha,\beta}(0)$. 

\begin{Def}\label{different thin notions}

Let $x_{0}\in\mathbb{R}^{n}$. 
\begin{itemize}
\item A set $E$ is called weakly $n$-thin at $x_{0}$, if 
\[
\sum^{\infty}_{i=1}{\rm cap}_{n}(E\cap A_{2^{-i-1},2^{-i}}(x_{0}),\,\,\mathring{A}_{2^{-i-2},2^{-i+1}}(x_{0}))<+\infty;
\]
\item A set $E$ is called weakly $n$-thin at infinity, if 
\[
\sum^{\infty}_{i=1}{\rm cap}_{n}(E\cap A_{2^{i},2^{i+1}},\,\,\mathring{A}_{2^{i-1},2^{i+2}})<+\infty;
\]
\item A set $E$ is called $n$-thin  at $x_{0}$, if 
\[
\sum^{\infty}_{i=1}i^{n-1}{\rm cap}_{n}(E\cap A_{2^{-i-1},2^{-i}}(x_{0}),\,\,\mathring{A}_{2^{-i-2},2^{-i+1}}(x_{0}))<+\infty;
\]
\item A set $E$ is called $n$-thin  at infinity, if 
\[
\sum^{\infty}_{i=1}i^{n-1}{\rm cap}_{n}(E\cap A_{2^{i},2^{i+1}},\,\,\mathring{A}_{2^{i-1},2^{i+2}})<+\infty;
\]
\item A set $E$ is called traditionally $n$-thin  at $x_{0}$, if 
\[
\sum^{\infty}_{i=1}{\rm cap}_{n}(E\cap B_{2^{-i}}(x_{0}),\,\,B_{2^{-i+1}}(x_{0}))^{\frac{1}{n-1}}<+\infty.
\]
\end{itemize}
\end{Def}

Notice that the traditionally $n$-thinness is the one used in \cite{KM1994}
or \cite{HKM1993}, to characterize an exceptional set $E$, such
that for an $n$-superharmonic function $v(x)$
\[
\lim_{x\to x_{0},x\notin E}v(x)=v(x_{0}).
\]
 While, $n$-thinness in our case, is used to characterize an exceptional
set $E$, such that for a positive $n$-superharmonic function $v(x)$
\[
\lim_{x\to x_{0},x\notin E}\frac{v(x)}{\log\frac{1}{|x|}}=m.
\]

\begin{Def}\label{strong e set}

A strong $\mathcal{E}$-set of $\mathbb{R}^{n}$ is a union of balls
$B_{i}=B(x_{i},r_{i})$
\[
E=\bigcup^{\infty}_{i=1}B_{i}
\]
such that $0\not\in\bar{B}_{i},i=1,\cdots,\infty$ and for some $\eta>0$
\[
\sum^{\infty}_{i=1}\left(\log_{+}\frac{|x_{i}|}{r_{i}}\right)^{1-n-\eta}<+\infty,
\]
where 
\[
\log_{+}x=\max\{\log x,0\}.
\]

\end{Def}

We compare different thinness notions and strong $\mathcal{E}$-set
notion.

\begin{lem}\label{thinness and strong epsilon comparison}

If $E$ is $n$-thin or traditionally $n$-thin at $x_{0}$, it is
weakly $n$-thin at $x_{0}$; If $E\subset B_{1}(0)$ is weakly $n$-thin
at $0$, it is contained in a strong $\mathcal{E}$-set. 

\end{lem}

\begin{proof}

Obviously, $n$-thinness implies weakly $n$-thinness at a point.
Suppose $E$ is traditionally $n$-thin at $x_{0}$, then 
\[
\sum^{\infty}_{i=1}\text{cap}_{n}(E\cap B_{2^{-i}}(x_{0}),B_{2^{-i+1}}(x_{0}))<+\infty,
\]
which implies 
\[
\sum^{\infty}_{i=1}\text{cap}_{n}(E\cap A_{2^{-i-1},2^{-i}}(x_{0}),B_{2^{-i+1}}(x_{0}))<+\infty.
\]
Since 
\[
\text{cap}_{n}(E\cap A_{2^{-i-1},2^{-i}}(x_{0}),B_{2^{-i+1}}(x_{0}))\simeq\text{cap}_{n}(E\cap A_{2^{-i-1},2^{-i}}(x_{0}),\mathring{A}_{2^{-i-2},2^{-i+1}}(x_{0})),
\]
it is weakly $n$-thin at $x_{0}$. Here by $A\simeq B$ we mean there
exists $C>0$ such that $C^{-1}A\le B\le CA.$ 

Suppose $F$ is weakly $n$-thin at $0$. From Lemma \ref{comparision of two capacities},
\begin{align*}
 & c_{1,n}(2^{i}(F\cap A_{2^{-i-1},2^{-i}}))=c_{1,n}(2^{i}F\cap A_{2^{-1},1})\simeq\text{cap}_{n}(2^{i}F\cap A_{2^{-1},1},B_{2})\\
\simeq & \text{cap}_{n}(2^{i}F\cap A_{2^{-1},1},\mathring{A}_{2^{-2},2})=\text{cap}_{n}(F\cap A_{2^{-i-1},2^{-i}},\mathring{A}_{2^{-i-2},2^{-i+1}}).
\end{align*}

On the other hand from \cite[Corollary 5.1.14]{AH1996}, for some
small $\eta>0$ and $C>0$, for a set $\tilde{F}$
\[
\Lambda^{(\infty)}_{h}(\tilde{F})^{\frac{n-1}{n-1+\eta}}\le Cc_{1,n}(\tilde{F}),
\]
where $h(r)=\left(\log_{+}\frac{2}{r}\right)^{1-n-\eta}$ and 
\[
\Lambda^{(\infty)}_{h}(\tilde{F})=\inf\{\sum_{i}h(r_{i});\bigcup B_{r_{i}}(x_{i})\supset\tilde{F}\}.
\]
Here we make a remark that, although \cite[Corollary 5.1.14]{AH1996}
is stated for compact subsets, due to \cite[Remark in Page 138]{AH1996},
it holds for arbitrary subsets. So taking $\tilde{F}=2^{i}(F\cap A_{2^{-i-1},2^{-i}})$,
\[
\Lambda^{(\infty)}_{h}(2^{i}(F\cap A_{2^{-i-1},2^{-i}}))^{\frac{n-1}{n-1+\eta}}\le C\text{cap}_{n}(F\cap A_{2^{-i-1},2^{-i}},\mathring{A}_{2^{-i-2},2^{-i+1}}).
\]
Since $F$ is weakly $n$-thin, 
\[
\sum^{\infty}_{i=1}\Lambda^{(\infty)}_{h}(2^{i}(F\cap A_{2^{-i-1},2^{-i}}))^{\frac{n-1}{n-1+\eta}}<+\infty.
\]
Then we may cover $F\cap A_{2^{-i-1},2^{-i}}$ by countable many balls
$B^{i}_{j}=B_{r^{i}_{j}}(x^{i}_{j})$ such that 
\[
\sum^{\infty}_{i=1}\left(\sum^{\infty}_{j=1}\left(\log_{+}\frac{2^{-i}}{r^{i}_{j}}\right)^{1-n-\eta}\right)^{\frac{n-1}{n-1+\eta}}<+\infty.
\]
Since we may assume $|x^{i}_{j}|\simeq2^{-i}$, and $\frac{n-1}{n-1+\eta}<1$
we have 
\[
\sum^{\infty}_{i,j=1}\left(\log_{+}\frac{|x^{i}_{j}|}{r^{i}_{j}}\right)^{1-n-\eta}<+\infty.
\]
By relabeling these balls as $B_{i}=B(x_{i},r_{i})$, we can rewrite
it as 
\[
\sum^{\infty}_{i=1}\left(\log_{+}\frac{|x_{i}|}{r_{i}}\right)^{1-n-\eta}<+\infty.
\]
Then $F$ is contained in a strong $\mathcal{E}$-set. 

\end{proof}

Crucially, the role of a strong $\mathcal{E}$-set is not to be smaller
in capacity or measure, but to ensure that outside such a set, the
function enjoys asymptotically exact radial symmetry with error $o(1)$,
while maintaining a well-behaved structure. This type of asymptotic
radial symmetry is substantially stronger and is indispensable for
the precise geometric applications in later sections.

\subsection{Main singularity estimates}

We revise \cite[Theorem 1.1]{MQ2021}. 

\begin{thm}\label{MQ2021 thm}

Let $u$ be a nonnegative $n$-superharmonic function in $B_{2}(0)\subset\mathbb{R}^{n}$
and 
\[
-\Delta_{n}u=\mu\ge0
\]
 for a Radon measure $\mu\ge0$. Then there is a set $E\subset\mathbb{R}^{n}$,
which is $n$-thin at $0$, such that 
\[
\lim_{x\notin E,x\to0}\frac{u(x)}{\log\frac{1}{|x|}}=\liminf_{x\to0}\frac{u(x)}{\log\frac{1}{|x|}}=m\ge0
\]
 and 
\[
u(x)\ge m\log\frac{1}{|x|}-C
\]
 for $x\in B_{1}(0)\backslash\{0\}$ and some $C$. 

\end{thm}

Now we state the main singularity estimates.

\begin{thm}\label{main thm 3-weak thin version}

Let $u\ge0$ be $n$-superharmonic in $B_{2}(0)\subset\mathbb{R}^{n}$
with
\[
-\Delta_{n}u=\mu
\]
 for a nonnegative Radon measure $\mu$. Then there is a weakly $n$-thin
set $F$, such that 
\[
u(x)=\underline{u}(|x|)+o(1),x\notin F
\]
where 
\[
\underline{u}(r)=\inf_{|x|=r}u(x).
\]

\end{thm}

Then apply Lemma \ref{thinness and strong epsilon comparison}, Theorem
\ref{main thm1 refined analytic result of u} holds.

\begin{cor}

Suppose that $(\mathbb{R}^{n},g=e^{2w}|dx|^{2})$ is complete with
nonnegative Ricci. Then there is a strong $\mathcal{E}$-set $E\subset\mathbb{R}^{n}$,
such that 
\[
w(x)=\underline{w}(|x|)+o(1),\,{\rm as}\,x\to\infty,x\notin E.
\]

\end{cor}

\begin{proof}

We let 
\[
u(y)=w(\frac{y}{|y|^{2}})-2\log|y|,y\in B_{2}(0).
\]
 Then $g=e^{2w}|dx|^{2}=e^{2u(y)}|dy|^{2},$ where $x=\frac{y}{|y|^{2}.}$
Then 
\[
-\Delta^{y}_{n}u(y)=Ric_{g}(\frac{\nabla^{g}u}{|\nabla^{g}u|_{g}},\frac{\nabla^{g}u}{|\nabla^{g}u|_{g}})|\nabla_{y}u|^{n-2}e^{2u}\ge0.
\]
And from \cite[Prop8.1]{CHY2004}, 
\[
\lim_{y\to0}u(y)=+\infty.
\]
 Then from Theorem \ref{main thm1 refined analytic result of u},
we know 
\[
u(y)=\underline{u}(y)+o(1),y\to0,y\notin F
\]
for some strong $\mathcal{E}$-set $F$. Since $\underline{u}(y)=\underline{w}(\frac{y}{|y|^{2}})-2\log|y|$,
and it is easy to check that $F$ is transformed into another strong
$\mathcal{E}$-set $E$ under inversion. Then we proved the corollary.

\end{proof}

\subsection{Radial properties of $n$-superharmonic functions}

To prove Theorem \ref{main thm 3-weak thin version}, we next establish
regularity and monotonicity properties of $\underline{u}(r)$. 

\begin{lem}\label{radial symmetric n-superharmonic functions}

Let $f(|x|)\ge0$ be $n$-superharmonic in $0<|x|<\delta$. Then
\begin{itemize}
\item $f(r)$ is continuous for $r\in(0,\delta)$;
\item $f'(r\pm0)$ exists for every $r\in(0,\delta)$, and $f'(r-0)\ge f'(r+0)$
with strict inequality at most countably often;
\item $f''(r)$ is a Radon measure in $(0,\delta)$;
\item $rf'(r\pm0)$ is non-increasing in $r$;
\item 
\begin{align*}
\lim_{r\to0}\frac{f(r)}{\log\frac{1}{r}} & =\lim_{r\to0}(-f'(r\pm0)r)=m\ge0;
\end{align*}
\item $f(|x|)$ is $n$-superharmonic in $|x|<\delta.$ 
\end{itemize}
\end{lem}

\begin{proof}

By comparison principle, $f(r)$ has the following logarithm concavity,
for $0<\alpha<r<\beta$, 
\[
f(r)\ge\frac{\log r-\log\beta}{\log\alpha-\log\beta}f(\alpha)+\frac{\log\alpha-\log r}{\log\alpha-\log\beta}f(\beta).
\]
Then if we denote $f(r)=g(\log r),$ $g(t)$ is concave in $t$. Then
$f(r)$ has similar regularity as a concave function does, which includes
the first three items. For the fourth item, simply note $rf'(r\pm0)=g'(\log r\pm0)$
is non-increasing in $r$. 

Then we may assume 
\[
\lim_{r\to0}(-rf'(r\pm0))=m\in[-\infty,+\infty).
\]
If $m<0$, there would be a contradiction with $f(r)>0$. So $m\ge0$.
Then from Lemma \ref{Generalized L'Hospital rule} below, we know
\[
\lim_{r\to0}\frac{f(r)}{\log\frac{1}{r}}=\lim_{r\to0}(-rf'(r\pm0))=m.
\]

For the last item, we notice that for $m\ge0$, $f'(r\pm0)\le0$.
Then $0$ is the local maximum point of $f(|x|)$. Thus, the $n$-superharmonicity
extends to the origin.

\end{proof}

\begin{lem}\cite[Theorem 6.3]{Gh2025}\label{Generalized L'Hospital rule}(Generalized
L'Hospital rule) Suppose that $-\infty\le a<b\le\infty$. Let $f$
and $g$ be real-valued functions on $(a,b)$ that are locally integrable
on $(a,b)$ such that $g>0$ a.e. on $(a,b)$. Suppose $h:(a,b)\to\mathbb{R}$
satisfies $f/g=h$ a.e. and 
\[
\lim_{x\to b}h(x)=A\in[-\infty,\infty].
\]
Then if $\int^{b}_{r}gdx=\infty$ $\forall r\in(a,b)$, then 
\[
\lim_{x\to b}\frac{\int^{x}_{r}fdx}{\int^{x}_{r}gdx}=A.
\]

\end{lem}

\begin{lem}\label{Lemma:u monotone.}Under the assumptions of Theorem
\ref{main thm 3-weak thin version}
\begin{itemize}
\item 
\begin{equation}
\lim_{r\to0}\frac{\underline{u}(r)}{\log\frac{1}{r}}=m=\liminf_{x\to0}\frac{u(x)}{\log\frac{1}{|x|}};\label{limit m}
\end{equation}
\item $\underline{u}(r)$ is continuous for $r>0;$ Either $\underline{u}(r)$
is strictly decreasing in $r$, or $\underline{u}(r)$ is constant
in some $(0,\delta)$; In particular
\[
\lim_{r\to0^{+}}\underline{u}(r)\ge0;
\]
\item $\underline{u}(r)-\underline{u}(5r)\to m\log5$ as $r\to0$. 
\end{itemize}
\end{lem}

\begin{proof}

Item 1 and 2 follow directly from Lemma \ref{radial symmetric n-superharmonic functions}.

For the third item, we consider, for $r>0$ small, 
\[
\underline{u}(r|\xi|)-\underline{u}(5r)
\]
which is $n$-superharmonic for $|\xi|\le5$. 
\[
-\Delta^{\xi}_{n}(\underline{u}(r|\xi|)-\underline{u}(5r))=\underline{\mu}_{r}(\xi)
\]
 with $\underline{\mu}_{r}(\Omega)=\underline{\mu}(r\Omega)$, where
$\underline{\mu}$ satisfies
\[
-\Delta^{x}_{n}\underline{u}(|x|)=\underline{\mu}(x).
\]
We know that $\underline{\mu}_{r}\to\underline{\mu}(\{0\})\delta(\xi),$
as $r\to0$. Notice that this convergence is in the ``converging
decomposition'' sense of \cite[(3.1)-(3.4)]{DMOP1999} or \cite[Definition 4.10]{Veron2017}.
Since $\underline{u}(r|\xi|)-\underline{u}(5r)=0$ on $\partial B_{5}$,
we can apply \cite[Theorem 3.2]{DMOP1999} or \cite[Theorem 4.3.8]{Veron2017}
to conclude that for any sequence $r_{i}\to0$
\[
\underline{u}(r_{i}|\xi|)-\underline{u}(5r_{i})
\]
subconverges to a renormalized solution, defined in \cite[Definition 2.13]{DMOP1999}
or \cite[Definition 4.9]{Veron2017}, $\Gamma(\xi)$ a.e. in $B_{5}(0)$
with 
\[
\begin{cases}
-\Delta^{\xi}_{n}\Gamma(\xi)=\underline{\mu}(\{0\})\delta(\xi), & \xi\in B_{5}(0)\\
\Gamma(\xi)|_{\partial B_{5}(0)}=0.
\end{cases}
\]
From \cite{Se1964,Se1965}, we know $\Gamma(\xi)\le C\log\frac{1}{|\xi|}+C$.
Then from the uniqueness result of \cite[Theorem 2.1]{KV1986}, we
know $\Gamma(\xi)=m\log\frac{5}{|\xi|}$, where $m=(\underline{\mu}(\{0\})/|\mathbb{S}^{n-1}|)^{1/(n-1)}$.
Moreover, the convergence is in $L^{p}(B_{5}(0))$ for $1\le p<n$.
Then it follows that $\underline{u}(rs)-\underline{u}(5r)$ converges
to $m\log\frac{5}{s}$ $L^{p}(B_{5}(0))$ as $r\to0$. Then
\begin{equation}
\int^{\frac{3}{2}}_{\frac{1}{2}}\left|\underline{u}(rs)-\underline{u}(5r)-m\log\frac{5}{s}\right|^{p}ds\to0,\,\,{\rm as}\,\,r\to0.\label{Lp convergence}
\end{equation}

For any $\varepsilon>0$, we choose $\delta_{1}>0$ such that when
$|s-1|<\delta_{1}$
\begin{equation}
\left|m\log\frac{5}{s}-m\log5\right|<\frac{\varepsilon}{2}.\label{log s - log 5}
\end{equation}
From (\ref{Lp convergence}), for $\varepsilon,\delta_{1}$ above,
there exists $\delta>0$ such that when $0<r<\delta$, 
\[
{\rm meas}\left\{ s\in\left[\frac{1}{2},\frac{3}{2}\right];\left|\underline{u}(rs)-\underline{u}(5r)-m\log\frac{5}{s}\right|>\frac{\varepsilon}{2}\right\} <\frac{\delta_{1}}{2}.
\]
Then when $0<r<\delta$, there exist $s_{1}\in(1-\frac{\delta_{1}}{2},1),s_{2}\in(1,1+\frac{\delta_{1}}{2})$
such that 
\[
\left|\underline{u}(rs_{1})-\underline{u}(5r)-m\log\frac{5}{s_{1}}\right|,\left|\underline{u}(rs_{2})-\underline{u}(5r)-m\log\frac{5}{s_{2}}\right|\le\frac{\varepsilon}{2}.
\]
So from (\ref{log s - log 5})
\[
\left|\underline{u}(rs_{1})-\underline{u}(5r)-m\log5\right|,\left|\underline{u}(rs_{2})-\underline{u}(5r)-m\log5\right|\le\varepsilon.
\]
Since $\underline{u}(rs)$ is monotonically decreasing in $s$, we
know 
\[
\left|\underline{u}(r)-\underline{u}(5r)-m\log5\right|\le\varepsilon.
\]

\end{proof}

Then one easy case of Theorem \ref{main thm1 refined analytic result of u}
can be proved.

\begin{thm}

Under the assumption of Theorem \ref{main thm1 refined analytic result of u}
and 
\[
\liminf_{x\to0}u(x)<+\infty,
\]
then 
\[
u(x)=\underline{u}(|x|)+o(1)
\]
for $x\to0,x\notin E$ for some traditionally $n$-thin set $E$. 

\end{thm}

\begin{proof}

From the assumption, there is a set $E$ which is traditionally $n$-thin
at $0$ such that 
\[
\lim_{x\to0,x\notin E}u(x)=\lim_{r\to0}\underline{u}(r)=u(0).
\]
Then we draw the conclusion. 

\end{proof}

\textbf{Below in this section, we assume}
\[
\liminf_{x\to0}u(x)=+\infty.
\]

\subsection{Estimates for scaled functions away from exceptional sets}

The proof of Theorem \ref{main thm 3-weak thin version} relies on
a compactness argument for scaled functions 
\[
u_{r}(\xi)=u(r\xi)-\underline{u}(5r),
\]
which converges strongly in $W^{1,p}_{loc}(\mathbb{R}^{n})$ for $1<p<n$
to $h(\xi)=m\log\frac{5}{|\xi|},$ where $m=(\mu(\{0\})/|\mathbb{S}^{n-1}|)^{1/(n-1)}$.
Moreover, the convergence is uniform in annuli $A_{1/2,1}\backslash E$,
where $E$ is a small set, whose size has control. In this subsection,
we do a priori estimate for $u_{r}(\xi)$ away from exceptional sets. 

\begin{lem}\label{u_r estimate in B_1}

There exists some constant $C_{1}>0$, independent of $r$ and $t\in(0,1]$,
such that
\[
u_{r}(\xi)\le C_{1}(\log\frac{1}{|\xi|}+1),\xi\in B_{1}(0)\backslash E_{r}
\]
 for some set $E_{r}$ which is $n$-thin at $0$. And moreover, for
fixed $t\in(0,1]$, $E_{r}$ has estimate
\[
\text{cap}_{n}(E_{r}\cap A_{2^{-1}t,t},\mathring{A}_{2^{-2}t,2t})\le C_{1}\mu_{r}(A_{2^{-2}t,2t})\to0,\,{\rm as}\,r\to0
\]
where 
\[
\mu_{r}=-\Delta^{\xi}_{n}u_{r}(\xi),
\]
which satisfies $\mu_{r}(\Omega)=\mu(r\Omega)$.

\end{lem}

\begin{proof}

It is easy to prove that 
\[
\mu_{r}\to\mu(\{0\})\delta(x)
\]
 in the sense of distribution. 

From Item 3 of Lemma \ref{Lemma:u monotone.} and $\underline{u}_{r}(5)=0$,
we know $\inf_{B_{1}(0)}u_{r}\le C.$ Then for $\xi\in B_{1}(0)$,
$\inf_{B_{2}(\xi)}u_{r}\le C.$ Then from \cite[Theorem 1.6]{KM1994},
\begin{align*}
u_{r}(\xi) & \le C(n)(\inf_{B_{2}(\xi)}u_{r}+W^{\mu_{r}}_{1,n}(\xi,3))\\
 & \le C(1+W^{\mu_{r}}_{1,n}(\xi,3)).
\end{align*}
Then we estimate $W^{\mu_{r}}_{1,n}(\xi,3)$, which is more or less
standard. By for example \cite[Theorem 2.3.1]{LMQZ2025}, we see
\[
W^{\mu_{r}}_{1,n}(\xi,3)\le C(\log\frac{1}{|\xi|}+1)
\]
 for $\xi\in B_{1}(0)\backslash E_{r}$ for a set $E_{r}$ $n$-thin
at $0$. Then we know that 
\[
u_{r}(\xi)\le C(\log\frac{1}{|\xi|}+1)
\]
 for $\xi\in B_{1}(0)\backslash E_{r}$ for a set $E_{r}$ $n$-thin
at $0$. 

Furthermore, from the proof of \cite[Theorem 2.3.1]{LMQZ2025},  we
know that 
\[
\text{cap}_{n}(E_{r}\cap A_{2^{-1}t,t},\mathring{A}_{2^{-2}t,2t})\le C_{1}\mu_{r}(A_{2^{-2}t,2t})\to0
\]
 as $r\to0$. 

\end{proof}

Now we estimate $\underline{u}_{r}$ in $B_{\frac{1}{5r}}(0)\backslash B_{1}(0)$
for small $r>0$.

\begin{lem}\label{=00005Cunderline u estimate}

For $\eta\in B_{\frac{1}{5r}}(0)\backslash B_{1}(0)$, 

\[
\left|\underline{u}_{r}(|\eta|)\right|\le C(\left|\log|\eta|\right|+1).
\]

\end{lem}

\begin{proof}

From Lemma \ref{u_r estimate in B_1}, for $0<s<1$
\[
\underline{u}_{r}(s)\le C(\log\frac{1}{s}+1).
\]
Now we suppose that $\eta\in B_{\frac{1}{5r}}(0)\backslash B_{1}(0).$
We know, for some $C>0$, 
\[
\underline{u}_{r|\eta|}(|\eta|^{-1})\le C(\left|\log|\eta|\right|+1).
\]
Since 
\[
\underline{u}_{r|\eta|}(1)\ge0
\]
 we know 
\[
\underline{u}_{r|\eta|}(|\eta|^{-1})-\underline{u}_{r|\eta|}(1)\le C(\left|\log|\eta|\right|+1)
\]
which implies
\[
\underline{u}_{r}(1)-\underline{u}_{r}(|\eta|)\le C(\left|\log|\eta|\right|+1).
\]
Since $\eta\ge1$, we know $\underline{u}_{r}(|\eta|)\le\underline{u}_{r}(1)\le C.$
Then from
\[
\underline{u}_{r}(1)>0
\]
we know 
\[
\left|\underline{u}_{r}(|\eta|)\right|\le C(\left|\log|\eta|\right|+1).
\]

\end{proof}

\begin{lem}\label{local estimate of w}

For $r$ small, there are a uniform constant $C_{2}$ and a set $E_{r}\subset A_{\frac{r}{2},r}$
such that 
\[
u(x)\le\inf_{A_{\frac{r}{2},r}}u+C_{2},\,x\in A_{\frac{r}{2},r}\backslash E_{r}
\]
 and 
\[
\text{cap}_{n}(E_{r},\mathring{A}_{\frac{r}{4},2r})\le C_{2}\mu(A_{\frac{r}{4},2r}).
\]

\end{lem}

\begin{proof}

The proof is a direct consequence of Lemma \ref{u_r estimate in B_1}.
Since 
\[
u_{r}(\xi)\le C\left(\log\frac{1}{|\xi|}+1\right)\le C,\forall\xi\in A_{\frac{1}{2},1}\backslash\hat{E}_{r}
\]
 for some $\hat{E}_{r}\subset A_{\frac{1}{2},1}$ and 
\[
\text{cap}_{n}(\hat{E}_{r},\mathring{A}_{\frac{1}{4},2})\le C\mu_{r}(A_{\frac{1}{4},2}).
\]
We choose $E_{r}=r\hat{E}_{r}$, and $x\in A_{\frac{r}{2},r}\backslash E_{r}\iff$$\xi=\frac{x}{r}\in A_{\frac{1}{2},1}\backslash\hat{E}_{r}$
\begin{align*}
u(x) & =u_{r}(\xi)+\underline{u}(5r)\le C+\underline{u}(5r)=C+\underline{u}(r)+\underline{u}(5r)-\underline{u}(r)\\
 & \le C+\inf_{A_{\frac{r}{2},r}}u,
\end{align*}
where the last inequality follows from Lemma \ref{Lemma:u monotone.}.

\end{proof}

\begin{lem}\label{global estimate for u_r_i}

Fix $R>0$. There exists $C>0$ which does not depend on $R$, $r$
and $t\in(0,R]$ such that, for $r>0$ small 
\[
\left|u_{r}(\xi)\right|\le C(\left|\log\frac{1}{|\xi|}\right|+1),\xi\in B_{R}\backslash\hat{E}_{r}
\]
outside a set $\hat{E}_{r}$ with the following estimate, for any
$t\in(0,R]$
\begin{equation}
\text{cap}_{n}(\hat{E}_{r}\cap A_{2^{-1}t,t},\mathring{A}_{2^{-2}t,2t})\le C\mu_{r}(A_{2^{-2}t,2t})=C\mu(A_{2^{-2}rt,2rt}).\label{capacity estimate-1}
\end{equation}

\end{lem}

\begin{proof}From Lemma \ref{u_r estimate in B_1}, we only need
to prove that, there exists $C_{1}>0$ such that 
\[
\left|u_{r}(\xi)\right|\le C_{1}(\left|\log\frac{1}{|\xi|}\right|+1),\xi\in A_{2^{-1},R}\backslash\tilde{E}_{r}
\]
and (\ref{capacity estimate-1}) holds for $\tilde{E}_{r}$ with $t\in[1,R]$.
Then we can let $\hat{E}_{r}=E_{r}\cup\tilde{E}_{r}$, where $E_{r}$
is defined in Lemma \ref{u_r estimate in B_1}.

From Lemma \ref{=00005Cunderline u estimate}, we know 
\[
u_{r}(\xi)\ge-C_{1}(\left|\log\frac{1}{|\xi|}\right|+1),\,\forall\xi\in A_{\frac{1}{2},R}.
\]

We define
\[
\tilde{E}_{r}=\{\xi\in A_{2^{-1},R};u_{r}(\xi)>\underline{u}_{r}(1)+C_{2}\}
\]
where $C_{2}$ is the same constant as in Lemma \ref{local estimate of w}. 

Then for any $t\in[1,R]$, and $\xi\in A_{2^{-1}t,t}\cap\tilde{E}_{r}$,
\begin{align*}
u_{r}(\xi) & >\underline{u}_{r}(1)+C\ge\inf_{A_{2^{-1}t,t}}u_{r}+C,
\end{align*}
So 
\[
\text{cap}_{n}(\tilde{E}_{r}\cap A_{2^{-1}t,t},\mathring{A}_{2^{-2}t,2t})\le C\mu_{r}(A_{2^{-2}t,2t}).
\]

\end{proof}

\subsection{$W^{1,p}$ convergence}

In this subsection, we prove that $u_{r}$ subconverges in $W^{1,p}_{loc}(\mathbb{R}^{n}),1<p<n,$
norm to $\hat{h}(\xi)=\left(\mu(\{0\})/|\mathbb{S}^{n-1}|\right)^{\frac{1}{n-1}}\log(1/|\xi|)+\beta$,
where $\mu=-\Delta_{n}u$ and $\beta$ is decided in the next subsection.
For this purpose, we need to employ some auxiliary functions. For
given $k>0$ and $R>0$ large, we let 
\begin{align*}
u_{r,R}(\xi) & =u_{r}(\xi)-\underline{u}_{r}(5R),\\
u^{k}_{r,R}(\xi) & =\min\{k,u_{r,R}(\xi)\},
\end{align*}
where $\underline{u}_{r}(s)=\inf_{|\xi|=s}u_{r}(\xi)$. 

\begin{lem}\label{u_=00007Br.R=00007D W^=00007B1.p=00007D norm bound}

For $1\le p<n$, there exists $C(p,R)>0$, for $r$ small,

\[
\|u_{r,R}(\xi)\|_{W^{1,p}(B_{2R})}\le C(p,R).
\]

\end{lem}

\begin{proof}

First note that 
\[
u_{r,R}(\xi)=u(r\xi)-\underline{u}(5rR).
\]
Since $R$ is fixed and $r$ can be very small, $rR$ is also very
small. Then from Lemma \ref{Lemma:u monotone.}, 
\[
u_{r,R}(\xi)\ge0,\forall\xi\in B_{5R}(0).
\]
 {}
\[
\inf_{B_{2R}(0)}u_{r,R}(\xi)=\underline{u}(2rR)-\underline{u}(5rR)\le m\log5+1
\]
 as $r\to0$. Here the readers are reminded that the case $m=0$ is
also included. Then from the weak Harnack inequality for positive
$p$-superharmonic functions, \cite[Theorem 3.51]{HKM1993}, we have
for any $q\ge1$, 
\[
\int_{B_{2R}}u_{r,R}(\xi)^{q}d\xi\le C(R,q)(\underline{u}(2rR)-\underline{u}(5rR))^{q}\le C(R,q).
\]
 Notice that \cite[Theorem 3.51]{HKM1993} is stated for $p$-supersolutions.
Nevertheless, it can apply to $p$-superharmonic function $u$ by
considering $p$-supersolutions $\min\{u,k\}$. 

Then from \cite[An alternative proof of Theorem 7.46]{HKM1993}, we
know for $1\le p<n$, 
\[
\int_{B_{2R}}|\nabla_{\xi}u_{r,R}|^{p}d\xi\le C(R,p).
\]
Then we finish the proof. 

\end{proof}

It is obvious that 
\[
-\Delta^{\xi}_{n}u_{r,R}(\xi)=-\Delta^{\xi}_{n}u_{r}(\xi)=\mu_{r}(\xi),\xi\in B_{5R},
\]
where $\mu_{r}(\Omega)=\mu(r\Omega)$. So $\mu_{r}\to\mu(\{0\})\delta(\xi)$
in the weak sense. We assume that 
\[
-\Delta^{\xi}_{n}u^{k}_{r,R}(\xi)=\mu^{k}_{r}(\xi).
\]

\begin{lem}\label{mu_r^k estimate}For $R>0$ fixed and $r$ small,
there exists $\rho$, which may depend on $r$ and $R$ such that
$\rho>5$, $\rho rR\to0$ as $r\to0$, and 

\[
\mu^{k}_{r}(B_{5R})\le C\mu_{r}(B_{2\rho R})=C\mu(B_{2\rho rR})\le C.
\]

\end{lem}

\begin{proof}

In the proof of this lemma, the $n$-thinness notion and fine topology
notion, under which convergence is allowed to hold outside the corresponding
thin set, are the traditional ones, which may be found in \cite[Chapter 12]{HKM1993}. 

Since $u_{r,R}$ is lower semicontinuous, $\{\xi;u_{r,R}(\xi)>0\}$
is an open subset. Since $0\in\{\xi;u_{r,R}(\xi)>0\}$, we choose
$\Omega_{1}$ as the connected component which contains $0$. It is
easy to see that $B_{5R}(0)\subset\Omega_{1}$. We first show that
$\Omega_{1}$ is bounded. 

Since 
\[
u_{r,R}(\xi)=u(r\xi)-\underline{u}(5rR),
\]
from Lemma \ref{local estimate of w}, for fixed $\rho>5$ with $\rho rR$
small, there is a subset $E_{2\rho rR}\subset A_{\rho rR,2\rho rR}$,
\begin{equation}
u(x)\le\underline{u}(2\rho rR)+C_{2},\,x\in A_{\rho rR,2\rho rR}\backslash E_{2\rho rR}\label{underline u 2rho r R}
\end{equation}
and 
\begin{equation}
\text{cap}_{n}(E_{2\rho rR},\mathring{A}_{\frac{\rho}{2}rR,4\rho rR})\le C_{2}\mu(A_{\frac{\rho}{2}rR,4\rho rR})\to0,\,{\rm as}\,r\to0.\label{n capacity go to 0}
\end{equation}

For fixed $R$, we may assume $rR$ small. Since we assume $\underline{u}(s)\to+\infty$
as $s\to0$ and $\underline{u}(x)$ is strictly monotonically decreasing,
we may choose $\rho\ge5$, which may depends on $r,R$, such that
\begin{align*}
\underline{u}(\rho rR)+C & <\underline{u}(5rR),\\
\lim_{r\to0}(\rho rR) & =0,
\end{align*}
 for $C=C_{2}$ given by (\ref{underline u 2rho r R}). Then 
\[
u(x)-\underline{u}(5rR)<0,\,x\in A_{\rho rR,2\rho rR}\backslash E_{2\rho rR}.
\]
Then from the capacity estimate (\ref{n capacity go to 0}), using
\cite[Lemma 5.1]{MQ2021}, we can find $s\in(\rho rR,2\rho rR)$ such
that
\[
\{|x|=s\}\cap E_{2\rho rR}=\emptyset.
\]
 Then 
\[
u_{r,R}(\xi)<0,\,{\rm for}\,|\xi|=sr^{-1}.
\]
So we know $\Omega_{1}\subset B_{sr^{-1}}(0)$ is bounded.

Second we show $n$-quasi everywhere on $\partial\Omega_{1}$
\[
u_{r,R}=0.
\]
We choose $x_{0}\in\partial\Omega_{1}$ such that $\Omega_{1}$ is
not traditionally $n$-thin at $x_{0}$. If $u_{r,R}(x_{0})>0$, then
there is a neighborhood $U_{x_{0}}$ of $x_{0}$ such that $u_{r,R}|_{U_{x_{0}}}>0$
which is a contradiction. If $u_{r,R}(x_{0})<0$, since the fine limit,
i.e. limit outside a traditionally thin set, 
\[
\finelim_{x\to x_{0}}u_{r,R}=u_{r,R}(x_{0})<0,
\]
we conclude that in any punctured neighborhood $(U_{x_{0}}\backslash x_{0})\cap\Omega_{1}$
there is some $x_{1}$ such that $u_{r,R}(x_{1})<0$. So we know $u_{r,R}(x_{0})=0$.
Then we know that $\forall x_{0}\in\partial\Omega$, 
\[
\finelim_{x\to x_{0}}u_{r,R}=0.
\]
So $u_{r,R}\in W^{1,p}_{0}(\Omega_{1})$ for any $1\le p<n$. Moreover
$u^{k}_{r,R}\in W^{1,n}_{0}(\Omega_{1})$ for any $k>0$ large. 

At last, we estimate $\mu^{k}_{r}$. For this we refer to \cite[Step 4]{Veron2017}.
We know $\mu^{k}_{r}$ has the decomposition in $\Omega_{1}$ as 
\[
\mu^{k}_{r}=\mu_{r}\lfloor\{u_{r,R}<k\}+\lambda^{k+}_{r,R},
\]
where $\lambda^{k+}_{r,R}$ is a nonnegative Radon measure supported
on $u_{r,R}=k$. Since $u^{k}_{r,R}\in W^{1,n}_{0}(\Omega_{1})$,
\[
\int_{\Omega_{1}}|\nabla_{\xi}u^{k}_{r,R}|^{n}dx=\int_{\{u_{r,R}<k\}}u_{r,R}d\mu_{r}+k\lambda^{k+}_{r,R}(\Omega_{1}).
\]
 Then from \cite[(4.3.25)]{Veron2017}, we have 
\[
\int_{\{u_{r,R}<k\}\cap\Omega_{1}}|\nabla_{\xi}u^{k}_{r,R}|^{n}dx\le Ck\|\mu_{r}(\Omega_{1})\|.
\]
Then we know 
\[
k\lambda^{k+}_{r,R}(\Omega_{1})\le Ck\|\mu_{r}(\Omega_{1})\|.
\]
Then $\mu^{k}_{r}(\Omega_{1})\le C\mu_{r}(\Omega_{1})$, which implies
the lemma. 

\end{proof}

Now we study the compactness of $u_{r,R}(\xi)$ in $W^{1,p}(B_{R})$
norm. We follow the proof of \cite{DMOP1999} or \cite[Theorem 4.3.8.]{Veron2017}. 

\begin{lem}\label{nabla u^k_=00007Br.R=00007D is a Cauchy sequence in measure}

There is a subsequence $r_{i}\to0$, such that $\nabla u^{k}_{r_{i},R}$
is a Cauchy sequence in measure, in $B_{R}(0)$. 

\end{lem}

\begin{proof}

For $\zeta\in C^{\infty}_{0}(B_{4R}),0\le\zeta\le1,$ $\zeta=1$ in
$B_{2R}(0)$ and $|\nabla\zeta|\le CR^{-1}$, we let
\begin{align*}
\psi & =\zeta^{n}u^{k}_{r_{i},R},\\
\nabla\psi & =\zeta^{n}\nabla u^{k}_{r_{i},R}+n\zeta^{n-1}u^{k}_{r_{i},R}\nabla\zeta.
\end{align*}
From 
\[
\int_{B_{4R}(0)}|\nabla u^{k}_{r_{i},R}|^{n-2}\nabla u^{k}_{r_{i},R}\nabla\psi=\int\psi d\mu^{k}_{r_{i}},
\]
 we have 
\[
\int\zeta^{n}|\nabla u^{k}_{r_{i},R}|^{n}+n\int\zeta^{n-1}u^{k}_{r_{i},R}|\nabla u^{k}_{r_{i},R}|^{n-2}\nabla u^{k}_{r_{i},R}\nabla\zeta\le k\cdot\mu^{k}_{r_{i}}(B_{4R}(0)).
\]
Then 
\[
\int\zeta^{n}|\nabla u^{k}_{r_{i},R}|^{n}\le C\int|\nabla\zeta|^{n}(u^{k}_{r_{i},R})^{n}+k\cdot\mu^{k}_{r_{i}}(B_{4R}(0)).
\]
So we have $\|\nabla u^{k}_{r_{i},R}\|_{L^{n}(B_{2R})}\le C(R,n,k,\mu(\{0\}))$.

We let $\eta>0$ to be fixed for defining the convergence in measure.
For $\gamma>0$ small and $\kappa>0$ large we set subsets of $B_{R}$
\begin{align*}
E_{1}= & \left\{ \sup\left\{ \left|\nabla u^{k}_{r_{i},R}\right|,\left|\nabla u^{k}_{r_{j},R}\right|\right\} >\kappa\right\} ,\\
E_{2}= & \left\{ \left|u^{k}_{r_{i},R}-u^{k}_{r_{j},R}\right|>\gamma\right\} ,\\
E_{3}= & \left\{ \sup\left\{ \left|\nabla u^{k}_{r_{i},R}\right|,\left|\nabla u^{k}_{r_{j},R}\right|\right\} \le\kappa\right\} \\
 & \cap\left\{ \left|u^{k}_{r_{i},R}-u^{k}_{r_{j},R}\right|\le\gamma\right\} \\
 & \cap\left\{ \left|\nabla u^{k}_{r_{i},R}-\nabla u^{k}_{r_{j},R}\right|>\eta\right\} .
\end{align*}
Then 
\[
\left\{ x\in B_{R};\left|\nabla u^{k}_{r_{i},R}-\nabla u^{k}_{r_{j},R}\right|>\eta\right\} \subset E_{1}\cup E_{2}\cup E_{3}.
\]
Since $u^{k}_{r_{i},R},u^{k}_{r_{j},R}$ are bounded in $W^{1,n}(B_{2R})$,
we may fix $\kappa>0$ large, such that 
\[
{\rm meas}(E_{1})\le\frac{\eta}{3}.
\]
 Then we estimate ${\rm meas}(E_{3})$. We let $0\le\phi\in C^{\infty}_{0}(B_{2R})$
such that $0\le\phi\le1$ and $\phi\equiv1$ in $B_{R}$, $|\nabla\phi|\le CR^{-1}$.
We define 
\[
T_{\gamma}f=\begin{cases}
f, & {\rm if}\,|f|\le\gamma,\\
{\rm (sgn}f)\gamma & {\rm if}\,|f|>\gamma.
\end{cases}
\]
Then 
\begin{align*}
 & \int_{B_{2R}(0)}\left(-\Delta^{\xi}_{n}u^{k}_{r_{i},R}+\Delta^{\xi}_{n}u^{k}_{r_{j},R}\right)\phi\cdot T_{\gamma}\left(u^{k}_{r_{i},R}-u^{k}_{r_{j},R}\right)\\
= & \int_{B_{2R}(0)}\phi\left(\left|\nabla u^{k}_{r_{i},R}\right|^{n-2}\nabla u^{k}_{r_{i},R}-\left|\nabla u^{k}_{r_{j},R}\right|^{n-2}\nabla u^{k}_{r_{j},R}\right)\nabla\left(u^{k}_{r_{i},R}-u^{k}_{r_{j},R}\right)\\
 & +T_{\gamma}\left(u^{k}_{r_{i},R}-u^{k}_{r_{j},R}\right)\left(\left|\nabla u^{k}_{r_{i},R}\right|^{n-2}\nabla u^{k}_{r_{i},R}-\left|\nabla u^{k}_{r_{j},R}\right|^{n-2}\nabla u^{k}_{r_{j},R}\right)\cdot\nabla\phi.
\end{align*}

So we have 
\begin{align*}
 & \eta^{n}{\rm meas}\left(E_{3}\right)\\
\le & \int_{E_{3}}\left|\nabla\left(u^{k}_{r_{i},R}-u^{k}_{r_{j},R}\right)\right|^{n}\\
\le & \int_{B_{2R}(0)}\phi\left(\left|\nabla u^{k}_{r_{i},R}\right|^{n-2}\nabla u^{k}_{r_{i},R}-\left|\nabla u^{k}_{r_{j},R}\right|^{n-2}\nabla u^{k}_{r_{j},R}\right)\nabla\left(u^{k}_{r_{i},R}-u^{k}_{r_{j},R}\right)\\
\le & \gamma\left(\mu^{k}_{r_{i}}\left(B_{2R}\right)+\mu^{k}_{r_{j}}\left(B_{2R}\right)\right)+\frac{C}{R}\gamma\left(\|u^{k}_{r_{i},R}\|_{W^{1,n-1}\left(B_{2R}\right)}+\|u^{k}_{r_{j},R}\|_{W^{1,n-1}\left(B_{2R}\right)}\right).
\end{align*}
From Lemma \ref{mu_r^k estimate}, we know $\mu^{k}_{r_{i}}\left(B_{2R}\right)$
are uniformly bounded. So we can choose $\gamma$ small such that
${\rm meas}(E_{3})\le\eta/3.$ 

In the end, we estimate ${\rm meas}(E_{2})$. Since $\nabla u^{k}_{r_{i},R}$
subconverges weakly in $L^{n}(B_{2R})$, $u^{k}_{r_{i},R}$ subconverges
strongly in $L^{n}(B_{2R})$. Then we know $u^{k}_{r_{i},R}$ is a
Cauchy sequence in measure. Then we can choose $i,j$ sufficiently
large such that 
\[
{\rm meas}(E_{2})\le\eta/3.
\]
 Then we finish the whole proof. 

\end{proof}

\begin{lem}\label{nabla u_=00007Br.R=00007D converges in measure}

There is a subsequence $r_{i}$ such that, $\{\nabla u_{r_{i},R}\}$
is a Cauchy sequence in measure in $B_{R}(0)$.

\end{lem}

\begin{proof}

Let $\sigma>0$, then
\[
\{|\nabla u_{r_{i},R}-\nabla u_{r_{j},R}|>\sigma\}\subset\{|\nabla u^{k}_{r_{i},R}-\nabla u^{k}_{r_{j},R}|>\sigma\}\cup\{u_{r_{i},R}>k\}\cup\{u_{r_{j},R}>k\}.
\]
 From Lemma \ref{nabla u^k_=00007Br.R=00007D is a Cauchy sequence in measure},
we know for given $\eta>0$, for $i,j$ large 
\[
{\rm meas}\{|\nabla u^{k}_{r_{i},R}-\nabla u^{k}_{r_{j},R}|>\sigma\}<\eta/3.
\]
 On the other hand, from \cite[Proposition 4.3.9]{Veron2017}, for
any $r>0$ there exists $C$ such that
\[
{\rm meas}\{\xi\in\Omega_{1};u_{r_{i},R}>k\}\le\frac{C\mu_{r_{i}}(\Omega_{1})}{k^{n-1}}
\]
 for any $\forall k\ge0$. 

Then we proved the lemma.

\end{proof}

\begin{lem}

\label{a convergence result}

For $1\le p<n$, $u_{r_{i}}$ subconverges in $W^{1,p}_{loc}(\mathbb{R}^{n})$
to 
\[
\hat{h}(\xi)=\left(\frac{\mu(\{0\})}{\left|\mathbb{S}^{n-1}\right|}\right)^{\frac{1}{n-1}}\log\frac{1}{|\xi|}+\beta.
\]

\end{lem}

\begin{proof}

Since 
\[
u_{r_{i}}(\xi)=u_{r_{i},R}(\xi)+\underline{u}_{r_{i}}(5R),
\]
from Lemma \ref{u_=00007Br.R=00007D W^=00007B1.p=00007D norm bound}
we know $\|u_{r_{i}}\|_{W^{1,p}(B_{2R})}\le C(p,R).$ Notice for $p<p'<n$,
uniform $W^{1,p'}$ bound gives the uniform integrability of $|\nabla u_{r_{i}}|$.
Then we use Vitali's convergence theorem to conclude that $u_{r_{i}}$
subconverges in $W^{1,p}(B_{R})$ norm to some function $\hat{h}(\xi)$.
Since we can choose $R=R_{i}\to+\infty$, we may find $u_{r_{i}}$
converges in $W^{1,p}_{loc}(\mathbb{R}^{n})$ to $\hat{h}(\xi)$. 

It is straightforward to check that, for $\phi\in C^{\infty}_{0}(B_{R})$,
\[
\int_{B_{R}}|\nabla u_{r_{i}}|^{n-2}\nabla u_{r_{i}}\nabla\phi\to\int_{B_{R}}|\nabla\hat{h}|^{n-2}\nabla\hat{h}\nabla\phi.
\]
Hence we have 
\begin{equation}
-\Delta^{\xi}_{n}\hat{h}(\xi)=\mu(\{0\})\delta(\xi).\label{h(=00005Cxi) equation}
\end{equation}
and $\hat{h}(\xi)$ is $n$-harmonic in $\mathbb{R}^{n}\backslash\{0\}$,
which has a continuous representative. 

Notice that passing to a subsequence $r_{i+1}\le2^{-4}r_{i}$, we
may assume $u_{r_{i}}(\xi)\to\hat{h}(\xi)$ almost everywhere. Notice
that from Lemma \ref{global estimate for u_r_i}, we may assume in
any $A_{\frac{l}{2},l},l\le R$,
\[
|\hat{h}(\xi)|\le C\left(\left|\log\frac{1}{|\xi|}\right|+1\right),\,\xi\in A_{\frac{l}{2},l}\backslash\bigcap_{j\ge1}\bigcup_{i\ge j}\hat{E}_{r_{i}}.
\]
And we have 
\begin{align*}
\text{cap}_{n}\left(\left(\bigcap_{j\ge1}\bigcup_{i\ge j}\hat{E}_{r_{i}}\right)\cap A_{\frac{l}{2},l},\mathring{A}_{\frac{l}{4},2l}\right) & \le\inf_{j\ge1}\sum_{i\ge j}\text{cap}_{n}\left(\hat{E}_{r_{i}}\cap A_{\frac{l}{2},l},\mathring{A}_{\frac{l}{4},2l}\right)\\
 & \le C\inf_{j\ge1}\sum_{i\ge j}\mu_{r_{i}}\left(A_{\frac{l}{4},2l}\right)\\
 & =C\inf_{j\ge1}\sum_{i\ge j}\mu\left(A_{\frac{r_{i}}{4}l},2r_{i}l\right)\\
 & \le C\inf_{j\ge1}\mu\left(B_{2r_{j}l}\backslash\{0\}\right)\\
 & =0.
\end{align*}
Then
\[
|\hat{h}(\xi)|\le C\left(\left|\log\frac{1}{|\xi|}\right|+1\right),\xi\in\mathbb{R}^{n}\backslash\{0\}.
\]
Then from \cite[Theorem 2.2]{KV1986}, we know
\[
\hat{h}(\xi)=\alpha\log\frac{1}{|\xi|}+\beta.
\]
Then from (\ref{h(=00005Cxi) equation}) 
\[
\hat{h}(\xi)=\left(\frac{\mu(\{0\})}{\left|\mathbb{S}^{n-1}\right|}\right)^{\frac{1}{n-1}}\log\frac{1}{|\xi|}+\beta.
\]

\end{proof}

\subsection{Uniqueness of the limit and local uniform convergence away from exceptional
set}

Lemma \ref{a convergence result} can also be applied to $\underline{u}(|x|)$,
which is also a nonnegative $n$-superharmonic function. Note that
\[
\underline{u}(r|\xi|)=\underline{u}_{r}(\xi)+\underline{u}(5r)
\]
and
\begin{align*}
-\Delta_{n}\underline{u}(|x|) & =\underline{\mu}(x),\\
-\Delta^{\xi}_{n}\underline{u}_{r}(|\xi|) & =\underline{\mu}_{r}(\xi).
\end{align*}

\begin{lem}\label{uniform convergence}

We have
\[
\underline{\mu}(\{0\})=\mu(\{0\})
\]
 and as $r\to0$

\[
\underline{u}_{r}(|\xi|)\to h(\xi)=\left(\frac{\mu(\{0\})}{\left|\mathbb{S}^{n-1}\right|}\right)^{\frac{1}{n-1}}\log\frac{5}{|\xi|}
\]
in $W^{1,p}_{loc}(\mathbb{R}^{n})$. Moreover, the convergence is
uniform in $A_{\frac{1}{2},6}$. 

\end{lem}

\begin{proof}

Apply Lemma \ref{a convergence result}, for any $\underline{u}_{r_{i}}(|\xi|),r_{i}\to0$,
we find a subsequence of $r_{i}$, still denoted by $r_{i}$, such
that $\underline{u}_{r_{i}}(|\xi|)$ converges to 
\[
\tilde{h}(\xi)=\left(\frac{\underline{\mu}(\{0\})}{\left|\mathbb{S}^{n-1}\right|}\right)^{\frac{1}{n-1}}\log\frac{1}{|\xi|}+\tilde{\beta}
\]
 a.e. and meanwhile $u_{r_{i}}(\xi)$ converges a.e. to 
\[
\left(\frac{\mu(\{0\})}{\left|\mathbb{S}^{n-1}\right|}\right)^{\frac{1}{n-1}}\log\frac{1}{|\xi|}+\beta.
\]
Then 
\[
\left(\frac{\underline{\mu}(\{0\})}{\left|\mathbb{S}^{n-1}\right|}\right)^{\frac{1}{n-1}}\log\frac{1}{|\xi|}+\tilde{\beta}\le\left(\frac{\mu(\{0\})}{\left|\mathbb{S}^{n-1}\right|}\right)^{\frac{1}{n-1}}\log\frac{1}{|\xi|}+\beta.
\]
So we must have $\underline{\mu}(\{0\})=\mu(\{0\})$ and $\tilde{\beta}\le\beta$.

As $\underline{u}_{r_{i}}(|\xi|)\to\tilde{h}(\xi)$ a.e. in $A_{\frac{1}{4},7}$,
by Egorov's theorem, for any $\varepsilon>0$, there is a subset $F\subset[\frac{1}{4},7]$
such that ${\rm meas}(F)<\varepsilon$, $\underline{u}_{r_{i}}(s)\to\left(\frac{\underline{\mu}(\{0\})}{\left|\mathbb{S}^{n-1}\right|}\right)^{\frac{1}{n-1}}\log\frac{1}{s}+\tilde{\beta}$
uniformly in $[\frac{1}{4},7]\backslash F$. Then we may choose $\delta>0$
such that when $r<\delta$, 
\[
\left|\underline{u}_{r_{i}}(s)-\left(\frac{\underline{\mu}(\{0\})}{\left|\mathbb{S}^{n-1}\right|}\right)^{\frac{1}{n-1}}\log\frac{1}{s}-\tilde{\beta}\right|\le\varepsilon,\forall s\in[\frac{1}{4},7]\backslash F.
\]

Now we consider when $s\in F\cap[\frac{1}{2},6]$. From Lemma \ref{Lemma:u monotone.},
we know $\underline{u}_{r_{i}}(s)$ is continuous and monotonically
decreasing.

If $s\in\partial F\cap[\frac{1}{2},6]$, then $s\in\partial F\cap(\frac{1}{4},7).$
Then there is $s_{n}\in[\frac{1}{4},7]\backslash F$ such that $s_{n}\to s$.
Since 
\[
\left|\underline{u}_{r_{i}}(s_{n})-\left(\frac{\underline{\mu}(\{0\})}{\left|\mathbb{S}^{n-1}\right|}\right)^{\frac{1}{n-1}}\log\frac{1}{s_{n}}-\tilde{\beta}\right|\le\varepsilon,
\]
we let $s_{n}\to s$, we have 
\[
\left|\underline{u}_{r_{i}}(s)-\left(\frac{\underline{\mu}(\{0\})}{\left|\mathbb{S}^{n-1}\right|}\right)^{\frac{1}{n-1}}\log\frac{1}{s}-\tilde{\beta}\right|\le\varepsilon.
\]

If $s\in\text{int}(F)\cap[\frac{1}{2},6]$, we let $s\in(s_{1},s_{2})\subset F$,
where $s_{1},s_{2}\in\partial F$. Since ${\rm meas}(F)$ is small,
we may assume $[s_{1},s_{2}]\subset(\frac{1}{4},7)$. There is a uniform
$C>0$ such that for $[s_{1},s_{2}]\subset(\frac{1}{4},7)$, 
\[
\osc_{[s_{1},s_{2}]}\left(\frac{\mu(\{0\})}{\left|\mathbb{S}^{n-1}\right|}\right)^{\frac{1}{n-1}}\log\frac{1}{|\xi|}\le C(s_{2}-s_{1}).
\]
Then 
\begin{align*}
 & \left|\underline{u}_{r_{i}}(s)-\left(\frac{\underline{\mu}(\{0\})}{\left|\mathbb{S}^{n-1}\right|}\right)^{\frac{1}{n-1}}\log\frac{1}{s}-\tilde{\beta}\right|\\
\le & \max\left\{ \left|\underline{u}_{r_{i}}(s_{1})-\left(\frac{\underline{\mu}(\{0\})}{\left|\mathbb{S}^{n-1}\right|}\right)^{\frac{1}{n-1}}\log\frac{1}{s_{2}}-\tilde{\beta}\right|,\left|\underline{u}_{r_{i}}(s_{2})-\left(\frac{\underline{\mu}(\{0\})}{\left|\mathbb{S}^{n-1}\right|}\right)^{\frac{1}{n-1}}\log\frac{1}{s_{1}}-\tilde{\beta}\right|\right\} \\
\le & \max\left\{ \left|\underline{u}_{r_{i}}(s_{1})-\left(\frac{\underline{\mu}(\{0\})}{\left|\mathbb{S}^{n-1}\right|}\right)^{\frac{1}{n-1}}\log\frac{1}{s_{2}}-\tilde{\beta}\right|+C\varepsilon,\left|\underline{u}_{r_{i}}(s_{2})-\left(\frac{\underline{\mu}(\{0\})}{\left|\mathbb{S}^{n-1}\right|}\right)^{\frac{1}{n-1}}\log\frac{1}{s_{1}}-\tilde{\beta}\right|+C\varepsilon\right\} \\
\le & C\varepsilon.
\end{align*}
So the convergence is uniform in $[\frac{1}{2},6]$. 

Since $\underline{u}_{r_{i}}(5)=0,$ we know 
\[
\tilde{h}(\xi)=\left(\frac{\mu(\{0\})}{\left|\mathbb{S}^{n-1}\right|}\right)^{\frac{1}{n-1}}\log\frac{5}{|\xi|}\overset{\bigtriangleup}{=}h(\xi).
\]
So the limit has uniqueness and then a.e. in $\mathbb{R}^{n}$ and
uniformly in $[\frac{1}{2},6]$, 
\[
\underline{u}_{r}(|\xi|)\to h(\xi).
\]

\end{proof}

Then we can prove the following main convergence result.

\begin{lem}\label{main convergence}

For $1\le p<n$, $u_{r}$ converges in $W^{1,p}_{loc}(\mathbb{R}^{n})$
to 
\[
h(\xi)=\left(\frac{\mu(\{0\})}{\left|\mathbb{S}^{n-1}\right|}\right)^{\frac{1}{n-1}}\log\frac{5}{|\xi|}.
\]

\end{lem}

\begin{proof}

For contradiction, we choose some $r_{i}\to0$, s.t. $u_{r_{i}}$
converges in $W^{1,p}_{loc}(\mathbb{R}^{n})$ to $h(\xi)+\varepsilon_{0},$
for some $\varepsilon_{0}>0$. Since $u_{r_{i}}$ is lower semi-continuous,
we find $|\hat{\xi}_{i}|=5$ s.t. $u_{r_{i}}(\hat{\xi}_{i})=\min_{\xi=5}u_{r_{i}}(\xi)=0.$
We assume $\hat{\xi}_{i}\to\hat{\xi}$. We find $\delta>0$ small
such that ${\rm meas}B_{\delta}(\hat{\xi})\le1$ and $\osc_{B_{\delta}(\hat{\xi})}h\le\varepsilon_{0}/M,$
for $M>10$ to be determined later. Since $u_{r_{i}}$ converges to
$h(\xi)+\varepsilon_{0}$ in $W^{1,p}(B_{\delta}(\hat{\xi}))$ and
$\underline{u}_{r_{i}}(\xi)$ converges to $h(\xi)$ uniformly in
$B_{\delta}(\hat{\xi})$, for $r_{i}$ small, 
\begin{align*}
 & \fint_{B_{\delta}(\hat{\xi})}\left|u_{r_{i}}(\xi)-\min_{B_{\delta}(\hat{\xi})}\underline{u}_{r_{i}}(|\xi|)\right|\\
\ge & \fint_{B_{\delta}(\hat{\xi})}\varepsilon_{0}-\fint_{B_{\delta}(\hat{\xi})}\left|u_{r_{i}}(\xi)-h(\xi)-\varepsilon_{0}\right|-\fint_{B_{\delta}(\hat{\xi})}\left|h(\xi)-\min_{B_{\delta}(\hat{\xi})}\underline{u}_{r_{i}}(|\xi|)\right|\\
\ge & \frac{\varepsilon_{0}}{2}.
\end{align*}
However, from weak Harnack inequality, \cite[Theorem 3.51]{HKM1993},
we have 
\begin{align*}
 & \fint_{B_{\delta}(\hat{\xi})}\left|u_{r_{i}}(\xi)-\min_{B_{\delta}(\hat{\xi})}\underline{u}_{r_{i}}(|\xi|)\right|\\
\le & C\inf_{B_{\frac{\delta}{2}}(\hat{\xi})}\left(u_{r_{i}}(\xi)-\min_{B_{\delta}(\hat{\xi})}\underline{u}_{r_{i}}(|\xi|)\right)\\
\le & C\left(u_{r_{i}}(\hat{\xi}_{i})-\min_{B_{\delta}(\hat{\xi})}\underline{u}_{r_{i}}(|\xi|)\right)\\
= & C\left(-\min_{B_{\delta}(\hat{\xi})}\underline{u}_{r_{i}}(|\xi|)\right)\\
\le & C\left|\underline{u}_{r_{i}}(5+\delta)-\left(\frac{\mu(\{0\})}{\left|\mathbb{S}^{n-1}\right|}\right)^{\frac{1}{n-1}}\log\frac{5}{5+\delta}\right|+C\left(\frac{\mu(\{0\})}{\left|\mathbb{S}^{n-1}\right|}\right)^{\frac{1}{n-1}}\left|\log\frac{5}{5+\delta}\right|.
\end{align*}
 As $r_{i}\to0$, we may assume 
\[
C\left|\underline{u}_{r_{i}}(5+\delta)-\left(\frac{\mu(\{0\})}{\left|\mathbb{S}^{n-1}\right|}\right)^{\frac{1}{n-1}}\log\frac{5}{5+\delta}\right|\le\frac{\varepsilon_{0}}{10}.
\]
For the second term 
\[
C\left(\frac{\mu(\{0\})}{\left|\mathbb{S}^{n-1}\right|}\right)^{\frac{1}{n-1}}\left|\log\frac{5}{5+\delta}\right|\le C\osc_{B_{\delta}(\hat{\xi})}h(\xi)\le\frac{C\varepsilon_{0}}{M}.
\]
We may choose $M$ big such that $C/M<\frac{1}{10}$. Then we get
a contradiction. 

\end{proof}

Next we estimate $u_{r}(\xi)$ further. For this we need to generalize
\cite[Theorem 3.9]{KM1994} a little.

\begin{lem}\label{capacity estimate in terms of measure}

( \cite[Lemma 2.2.2]{LMQZ2025})

Suppose that $u\ge0$ is $p$-superharmonic in $\Omega$ which satisfies
$-\Delta_{p}u=\mu,1<p\le n$, $\mu$ is a nonnegative Radon measure
which has compact support in $\Omega$ and $\min\{u,\lambda\}\in W^{1,p}_{0}(\Omega)$.
Then for $\lambda>0$ there holds
\[
\lambda^{p-1}{\rm cap}_{p}(\{x\in\Omega;u(x)>\lambda\},\Omega)\le\mu(\Omega).
\]

\end{lem}

One may find the proof in \cite[Lemma 2.2.2]{LMQZ2025}. 

\begin{lem}\label{u_r estimate further}

For any $\varepsilon>0$, there exist a uniform constant $C$, and
$\delta(\varepsilon)>0$ such that if $0<r<\delta$, 
\[
\left|u_{r}(\xi)-h(\xi)\right|\le C\varepsilon,\,\forall\xi\in A_{\frac{1}{2},1}\backslash\grave{E}_{r}
\]
and 
\[
\text{cap}_{n}(\grave{E}_{r},\mathring{A}_{\frac{1}{4},2})\le C(\varepsilon)\mu_{r}(A_{\frac{1}{4},2}).
\]

\end{lem}

\begin{proof}

For any $\varepsilon>0$, we use $l(\varepsilon)\in\mathbb{N}$ balls
$B_{1},\cdots,B_{l}$ to cover $\bar{A}_{\frac{1}{2},1}$ and the
concentric balls $10B_{i}\subset A_{\frac{1}{4},2}.$ We may also
choose the radius $\text{rad}(B_{i})\le\kappa\varepsilon$ for some
uniform $\kappa>0$ such that ${\rm osc}_{10B_{i}}h(\xi)\le\varepsilon/2,\forall i=1,\cdots,l$. 

From Lemma \ref{main convergence}, $u_{r}(\xi)\to h(\xi)$ in $W^{1,p}(A_{\frac{1}{4},2})$.
For $B=B_{i}$ or $6B_{i}$, $\inf_{B}u_{r}\ge\inf_{B}\underline{u}_{r}\to\inf_{B}h(\xi)$
and $\|u_{r}(\xi)-h(\xi)\|_{L^{1}(B)}\to0$, as $r\to0$,  we can
choose $\delta>0$ small such that when $0<r<\delta$, 
\[
\left|\inf_{B_{i}}u_{r}(\xi)-\inf_{B_{i}}h(\xi)\right|\le\varepsilon,\,\,\left|\inf_{6B_{i}}u_{r}(\xi)-\inf_{6B_{i}}h(\xi)\right|\le\varepsilon.
\]
Then from \cite[Theorem 1.6]{KM1994}, for any $\xi\in B_{i}$, we
let $r_{i}=\text{rad}(B_{i})$
\begin{align*}
u_{r}(\xi)-\inf_{6B_{i}}u_{r}(\xi) & \le C(n)\left(\inf_{B_{2r_{i}}(\xi)}u_{r}(\xi)-\inf_{6B_{i}}u_{r}(\xi)+W^{\mu_{r}}_{1,n}(\xi,4r_{i})\right)\\
 & \le C(n)\left(\inf_{B_{i}}h(\xi)+\varepsilon-\inf_{6B_{i}}u_{r}(\xi)+W^{\mu_{r}}_{1,n}(\xi,4r_{i})\right).
\end{align*}

So 
\begin{align*}
u_{r}(\xi)-\inf_{6B_{i}}u_{r}(\xi) & \le C(n)\left(\inf_{B_{i}}h(\xi)-\inf_{6B_{i}}h(\xi)+C\varepsilon+W^{\mu_{r}}_{1,n}(\xi,4r_{i})\right)\\
 & \le C(\varepsilon+W^{\mu_{r}}_{1,n}(\xi,4r_{i})).
\end{align*}

Now we solve 
\[
\begin{cases}
-\Delta^{\xi}_{n}v(\xi) & =\nu_{r},\,x\in8B_{i}\\
v(\xi)|_{\partial(8B_{i})} & =0,
\end{cases}
\]
where $\nu_{r}=\mu_{r}\lfloor6B_{i}$. From \cite[Theorem 1.6]{KM1994},
we know 
\[
W^{\mu_{r}}_{1,n}(\xi,4r_{i})\le C(n)v(\xi).
\]
Then 
\[
u_{r}(\xi)-\inf_{6B_{i}}u_{r}(\xi)\le C\varepsilon+Cv(\xi).
\]
Then 
\begin{align*}
\left|u_{r}(\xi)-h(\xi)\right| & \le\left|u_{r}(\xi)-\inf_{6B_{i}}u_{r}\right|+\left|\inf_{6B_{i}}u_{r}-h(\xi)\right|\\
 & \le C_{3}\varepsilon+C_{3}v(\xi).
\end{align*}
From Lemma \ref{capacity estimate in terms of measure}, 
\[
\text{cap}_{n}\left(\left\{ \xi\in8B_{i};v(\xi)>\varepsilon\right\} ,8B_{i}\right)\le\frac{\nu_{r}(8B_{i})}{\varepsilon^{n-1}}=\frac{\mu_{r}(6B_{i})}{\varepsilon^{n-1}}.
\]
Then 
\[
\text{cap}_{n}\left(\left\{ \xi\in B_{i};\left|u_{r}(\xi)-h(\xi)\right|>2C_{3}\varepsilon\right\} ,8B_{i}\right)\le\frac{\mu_{r}(6B_{i})}{\varepsilon^{n-1}}.
\]
So 
\begin{align*}
\text{cap}_{n}\left(\left\{ \xi\in A_{\frac{1}{2},1};\left|u_{r}(\xi)-h(\xi)\right|>2C_{3}\varepsilon\right\} ,\mathring{A}_{\frac{1}{4},2}\right) & \le\frac{\sum^{l}_{i=1}\mu_{r}(6B_{i})}{\varepsilon^{n-1}}\\
 & \le C(\varepsilon)\mu_{r}(A_{\frac{1}{4},2}).
\end{align*}

\end{proof}

\subsection{Proof of Theorem \ref{main thm 3-weak thin version}}

Lemma \ref{u_r estimate further} implies the asymptotical radial
symmetry outside weakly $n$-thin set. 

\begin{lem}

Under the assumption of Theorem \ref{main thm 3-weak thin version}, 

\[
u(x)=\underline{u}(|x|)+o(1)
\]
for $x\to0,x\notin E$ for some $E$ which is weakly $n$-thin at
$0$. 

\end{lem}

\begin{proof}

We choose a sequence $0<\varepsilon_{n}\to0$. From Lemma \ref{u_r estimate further},
for any $\varepsilon_{n}>0$ there exist uniform constant $C$ and
$2^{-i_{n}}$ such that if $0<r<2^{-i_{n}}$
\begin{align*}
\left|u_{r}(\xi)-h(\xi)\right| & \le C\varepsilon_{n},\forall\xi\in A_{\frac{1}{2},1}\backslash\grave{E}_{r}\\
\left|\underline{u}_{r}(\xi)-h(\xi)\right| & \le\varepsilon_{n},\forall\xi\in A_{\frac{1}{2},1}
\end{align*}
and 
\[
\text{cap}_{n}(\grave{E}_{r},\mathring{A}_{\frac{1}{4},2})\le C(\varepsilon_{n})\mu_{r}(A_{\frac{1}{4},2}).
\]
Hence we know 
\[
\left|u_{r}(\xi)-\underline{u}_{r}(\xi)\right|\le C\varepsilon_{n},\forall\xi\in A_{\frac{1}{2},1}\backslash\grave{E}_{r}.
\]
Then for $x\in B_{2^{-i_{n}}}$, we assume $2^{-i-1}\le|x|<2^{-i},i\ge i_{n}$,
and let $\xi=2^{i}x$. We have 
\[
\left|u(x)-\underline{u}(|x|)\right|\le C\varepsilon_{n},\forall\xi\in A_{2^{-i-1},2^{-i}}\backslash E^{\varepsilon_{n}}_{i}
\]
where $E^{\varepsilon_{n}}_{i}\subset A_{2^{-i-1},2^{-i}}$ satisfies
\begin{align*}
\text{cap}_{n}(E^{\varepsilon_{n}}_{i},\mathring{A}_{2^{-i-2},2^{-i+1}}) & \le C(\varepsilon_{n})\mu_{2^{-i}}(A_{\frac{1}{4},2})\\
 & =C(\varepsilon_{n})\mu(A_{2^{-i-2},2^{-i+1}}),
\end{align*}
which implies 
\[
\sum_{i\ge i_{n}}\text{cap}_{n}(E^{\varepsilon_{n}}_{i},\mathring{A}_{2^{-i-2},2^{-i+1}})<+\infty.
\]
Then, we find $j_{n}\ge i_{n}$ such that 
\[
\sum_{i\ge j_{n}}\text{cap}_{n}(E^{\varepsilon_{n}}_{i},\mathring{A}_{2^{-i-2},2^{-i+1}})\le2^{-n}.
\]
 We require $j_{n+1}>j_{n}$ and define $E$ by 
\[
E\bigcap A_{2^{-j_{n+1}},2^{-j_{n}}}=\bigcup^{j_{n+1}-1}_{i\ge j_{n}}E^{\varepsilon_{n}}_{i}
\]
then $E$ is weakly $n$-thin at $0$ and 
\[
u(x)=\underline{u}(|x|)+o(1)
\]
if $x\not\in E$ and $x\to0$. 

\end{proof}

\begin{cor}

In Theorem \ref{MQ2021 thm} 
\[
m=\left(\frac{\mu(\{0\})}{|\mathbb{S}^{n-1}|}\right)^{\frac{1}{n-1}}.
\]

\end{cor}

\begin{proof}

Since 
\[
\lim_{s\to0}\frac{\underline{u}(s)}{\log\frac{1}{s}}=\liminf_{x\to0}\frac{u(x)}{\log\frac{1}{|x|}}=m
\]
and 
\[
-\Delta_{n}\underline{u}(|x|)=\underline{\mu}
\]
 with $\underline{\mu}(\{0\})=\mu(\{0\})$, it suffices to prove that
\[
m=\left(\frac{\underline{\mu}(\{0\})}{|\mathbb{S}^{n-1}|}\right)^{\frac{1}{n-1}}.
\]

From Lemma \ref{radial symmetric n-superharmonic functions}, $\underline{u}(s)$
is continuous and $\underline{u}'(s\pm0)$ exists for every $s\in(0,\delta)$
for some $\delta>0$ and $\underline{u}'(s-0)=\underline{u}'(s+0)$
holds except at countably many points $s_{i}$, where $\underline{u}'(s-0)>\underline{u}'(s+0)$.
It is straightforward to check, for Radon measure $\underline{\mu}=-div(|\nabla\underline{u}|^{n-2}\nabla\underline{u})$
\[
\underline{\mu}(B_{s}(0))=\int_{\partial B_{s}(0)}(-u'(s-0))^{n-1}.
\]
So 
\[
\lim_{s\to0}(-u'(s-0))^{n-1}s^{n-1}|\mathbb{S}^{n-1}|=\underline{\mu}(\{0\})=\mu(\{0\}),
\]
which implies, by Lemma \ref{radial symmetric n-superharmonic functions}
\[
u(s)=\left(\frac{\mu(\{0\})}{|\mathbb{S}^{n-1}|}\right)^{\frac{1}{n-1}}\log\frac{1}{s}+o(\log\frac{1}{s})
\]
 and hence we draw the conclusion. 

\end{proof}

\section{Volume growth ratio estimate}

We revise \cite[Theorem 1.3]{MQ2021}. 

\begin{thm}Suppose that $(\mathbb{R}^{n},e^{2w}|dx|^{2})$ is complete
with nonnegative Ricci. Then there is a subset $E\subset\mathbb{R}^{n}$,
which is $n$-thin at $\infty,$ such that 
\[
\lim_{x\notin E,x\to\infty}\frac{w(x)}{\log\frac{1}{|x|}}=-\liminf_{x\to\infty}\frac{w(x)}{\log|x|}=\tilde{m}
\]
 and 
\begin{equation}
w(x)\ge\tilde{m}\log\frac{1}{|x|}-C\label{w lower bound}
\end{equation}
for some constant $C$, where $\tilde{m}\in[0,1]$ and $\tilde{m}=0$
if and only if $g$ is flat. 

\end{thm}

As a consequence of (\ref{w lower bound}), we have the following
lemma.

\begin{lem}

Assume $w$ as in the above theorem and 
\[
-\Delta_{n}w=\mu,
\]
then $\mu(\mathbb{R}^{n})<+\infty$. 

\end{lem}

\begin{proof}

For $R\ge2$ set 
\[
z_{R}(x):=w(x)-\inf_{B_{2R}(0)}w.
\]
Then $z_{R}\ge0$ in $B_{2R}(0)$ and $-\Delta_{n}z_{R}=\mu$. The
standard lower Wolff-potential estimate gives 
\[
\int^{R}_{1}\mu(B_{t}(0))^{\frac{1}{n-1}}\frac{dt}{t}\le Cz_{R}(0).
\]
Then (\ref{w lower bound}) implies 
\[
z_{R}(0)=w(0)-\inf_{B_{2R}(0)}w\le C(1+\log R).
\]
Therefore, 
\[
\int^{R}_{1}\mu(B_{t}(0))^{\frac{1}{n-1}}\frac{dt}{t}\le C(1+\log R).
\]

For $1\le T<R$, 
\[
\mu(B_{T}(0))^{\frac{1}{n-1}}\log\frac{R}{T}\le C(1+\log R).
\]
Letting $R\to\infty$ gives a bound for $\mu(B_{T}(0))$ independent
of $T$: 
\[
\sup_{T>0}\mu(B_{T}(0))<\infty.
\]
By monotone convergence, 
\[
\mu(\mathbb{R}^{n})=\lim_{T\to\infty}\mu(B_{T}(0))<\infty.
\]

\end{proof}

Since $e^{2w}|dx|^{2}$ is complete, we may assume $0\le\tilde{m}\le1$.
This section relates the asymptotic exponent $\tilde{m}$ of the conformal
factor to the asymptotic volume ratio of the manifold. The key is
to use the asymptotic radial symmetry of $w$ outside a strong $\mathcal{E}$-set
established in Section 2 to show that the volume growth is exactly
determined by the radial infimum $\underline{w}(r)=\inf_{|x|=r}w(x).$
Now we prove Theorem \ref{main thm2 volume ratio}.

\begin{proof}

Notice that the asymptotic volume ratio $\nu$ is independent of the
center $p$ of the geodesic ball. So we may choose $p=0$. We let
\[
\underline{w}(r)=\inf_{|x|=r}w(x).
\]
We define a $C^{0}$ conformal metric as 
\[
\underline{g}=e^{2\underline{w}(|x|)}|dx|^{2}.
\]
 We use $B^{\underline{g}}_{l}(0)$ to denote 
\[
\left\{ x;\int^{|x|}_{0}e^{\underline{w}(s)}ds<l\right\} .
\]
We define the $\underline{g}$-volume by 
\[
V_{\underline{g}}(B^{\underline{g}}_{l}(0))=n\omega_{n}\int^{\alpha}_{0}e^{n\underline{w}(s)}s^{n-1}ds
\]
where $\alpha$ satisfies 
\begin{equation}
l=l(\alpha)=\int^{\alpha}_{0}e^{\underline{w}(s)}ds.\label{l(=00005Calpha)}
\end{equation}
First we calculate 
\[
\lim_{l\to\infty}\frac{V_{\underline{g}}(B^{\underline{g}}_{l}(0))}{\omega_{n}l^{n}}.
\]

By Lemma \ref{Generalized L'Hospital rule}, 
\begin{align*}
\lim_{l\to+\infty}\frac{V_{\underline{g}}(B^{\underline{g}}_{l}(0))}{\omega_{n}l^{n}} & =\lim_{\alpha\to+\infty}\frac{\int^{\alpha}_{0}e^{n\underline{w}(s)}n\omega_{n}s^{n-1}ds}{\omega_{n}l(\alpha)^{n}}\\
 & =\lim_{\alpha\to+\infty}\frac{e^{n\underline{w}(\alpha)}n\omega_{n}\alpha^{n-1}}{n\omega_{n}l^{n-1}e^{\underline{w}(\alpha)}}\\
 & =\left(\lim_{\alpha\to+\infty}\frac{e^{\underline{w}(\alpha)}\alpha}{l}\right)^{n-1}\\
 & =\left(\lim_{\alpha\to+\infty}\frac{e^{\underline{w}(\alpha)}+\alpha e^{\underline{w}(\alpha)}\underline{w}'(\alpha)}{e^{\underline{w}(\alpha)}}\right)^{n-1}\\
 & =\left(\lim_{\alpha\to+\infty}\left(1+\alpha\underline{w}'(\alpha)\right)\right)^{n-1}.
\end{align*}
To calculate $\lim_{\alpha\to+\infty}\alpha\underline{w}'(\alpha),$
by inversion $u(\frac{x}{|x|^{2}})=w(x)+2\log|x|$, we know $u(y)$
is $n$-superharmonic in $B_{\delta}(0)$ and $\underline{w}(x)+2\log|x|=\underline{u}(\frac{1}{|x|})$.
Then apply Lemma \ref{radial symmetric n-superharmonic functions}
to $\underline{u}$ and then back to $\underline{w}$, we get, 
\[
\lim_{\alpha\to+\infty}\alpha\underline{w}'(\alpha)=-\tilde{m}.
\]
Then 
\[
\lim_{l\to+\infty}\frac{V_{\underline{g}}(B^{\underline{g}}_{l}(0))}{\omega_{n}l^{n}}=\nu=(1-\tilde{m})^{n-1}.
\]

Then we compare $V_{\underline{g}}(B^{\underline{g}}_{l}(0))$ with
$V_{g}(B^{g}_{l}(0))$, for which we need to use Theorem \ref{main thm1 refined analytic result of u}.
For $x\in\partial B^{\underline{g}}_{l}(0)$, we find a minimizing
$g$-geodesic $\gamma$ which joins $0$ with $x$. We use $d(\cdot,\cdot)$
to denote the distance function and $Length(\cdot,\cdot)$ to denote
the length of a curve induced by a certain metric. We know 
\[
d_{g}(0,x)=Length(\gamma,g)\ge l.
\]
So $B^{g}_{l}(0)\subset B^{\underline{g}}_{l}(0)$ and if we define
$\alpha$ by (\ref{l(=00005Calpha)}), then 
\begin{align*}
V_{g}(B^{g}_{l}(0)) & \le V_{g}(B^{\underline{g}}_{l}(0))=\int_{B_{\alpha}(0)}e^{nw}dx.
\end{align*}
We will show that
\begin{equation}
\lim_{\alpha\to+\infty}\frac{\int_{B_{\alpha}(0)}e^{nw}dx}{\int_{B_{\alpha}(0)}e^{n\underline{w}}dx}=1.\label{volume ratio go to 1}
\end{equation}
Obviously, 
\[
\frac{\int_{B_{\alpha}(0)}e^{nw}dx}{\int_{B_{\alpha}(0)}e^{n\underline{w}}dx}\ge1.
\]
On the other hand, since $w=\underline{w}+o(1)$ for $x\in\mathbb{R}^{n}\backslash E$,
where $E$ is a strong $\mathcal{E}$-set, i.e. 
\[
E=\bigcup^{\infty}_{i=1}B_{i}
\]
where $B_{i}=B_{r_{i}}(x_{i})$ with 
\begin{equation}
\sum_{i}\left(\log_{+}\frac{|x_{i}|}{r_{i}}\right)^{1-n-\eta}<+\infty\label{Strong epsilon set}
\end{equation}
 for some $\eta>0$. 
\begin{align*}
\lim_{\alpha\to+\infty}\frac{\int_{B_{\alpha}(0)}e^{nw}dx}{\int_{B_{\alpha}(0)}e^{n\underline{w}}dx} & =\lim_{\alpha\to+\infty}\frac{\int_{B_{\alpha}(0)\backslash E}e^{nw}dx+\int_{B_{\alpha}(0)\cap E}e^{nw}dx}{\int_{B_{\alpha}(0)}e^{n\underline{w}}dx}.
\end{align*}

Notice that 
\[
\frac{\int_{B_{\alpha}(0)\backslash E}e^{nw}dx}{\int_{B_{\alpha}(0)}e^{n\underline{w}}dx}\le\frac{\int_{B_{\alpha}(0)}e^{n\underline{w}+o(1)}dx}{\int_{B_{\alpha}(0)}e^{n\underline{w}}dx}.
\]
 It is easy to construct a continuous function $\zeta(l)=o(1)$ as
$l\to+\infty$ such that 
\[
n\underline{w}(l)+o(1)\le n\underline{w}(l)+\zeta(l).
\]
\begin{align}
\lim_{\alpha\to+\infty}\frac{\int_{B_{\alpha}(0)\backslash E}e^{nw}dx}{\int_{B_{\alpha}(0)}e^{n\underline{w}}dx} & \le\lim_{\alpha\to+\infty}\frac{\int_{B_{\alpha}(0)}e^{n\underline{w}+\zeta(|x|)}dx}{\int_{B_{\alpha}(0)}e^{n\underline{w}}dx}\nonumber \\
 & =\lim_{\alpha\to+\infty}\frac{e^{n\underline{w}(\alpha)+\zeta(\alpha)}n\omega_{n}\alpha^{n-1}}{e^{n\underline{w}(\alpha)}n\omega_{n}\alpha^{n-1}}\nonumber \\
 & =1.\label{volume ratio limit less than 1}
\end{align}

Now we estimate 
\[
\lim_{l\to+\infty}\frac{\int_{B_{\alpha}(0)\cap E}e^{nw}dx}{\int_{B_{\alpha}(0)}e^{n\underline{w}}dx}.
\]
Since from \cite[Theorem 1.3]{MQ2021}, 
\[
\underline{w}\ge\tilde{m}\log\frac{1}{r}-C,\,\tilde{m}\in[0,1]
\]
we have 
\[
\int_{B_{\alpha}(0)}e^{n\underline{w}}dx\to+\infty,{\rm as}\,l\to\infty.
\]
 Then by using Stolz's theorem, it suffices to prove that, uniformly
for $\delta\in(\frac{1}{2},1]$, 
\[
\frac{\int_{A_{\delta2^{k-1},\delta2^{k}}\cap E}e^{nw}dx}{\int_{A_{\delta2^{k-1},\delta2^{k}}}e^{n\underline{w}}dx}\to0,\,{\rm as}\,k\to+\infty.
\]
Hereinafter, for brevity, the reader may ignore $\delta$ and this
does not affect the reading.

For $B_{i}=B_{r_{i}}(x_{i})$ which belongs to the strong $\mathcal{E}$-set
$E$ with $B_{i}\cap A_{\delta2^{k-1},\delta2^{k}}\ne\emptyset$,
we may assume $r_{i}\ll|x_{i}|\sim2^{k}$. We fixed some $\delta_{1}>0$
and from \cite[Theorem 1.6]{KM1994}, for any $x\in B_{r_{i}}(x_{i})$,
\[
w(x)-\inf_{B_{3\delta_{1}2^{k}}(x)}w\le C\left(\inf_{B_{\delta_{1}2^{k}}(x)}w-\inf_{B_{3\delta_{1}2^{k}}(x)}w\right)+CW^{\mu}_{1,n}(x,2\delta_{1}2^{k}).
\]
Given $\varepsilon>0$, since $r_{i}\ll\delta_{1}2^{k}$, for $k$
large, we may assume
\[
\left|\inf_{B_{\delta_{1}2^{k}}(x)}w-\inf_{B_{\delta_{1}2^{k}}(x)}\underline{w}\right|<\varepsilon,\,\left|\inf_{B_{3\delta_{1}2^{k}}(x)}w-\inf_{B_{3\delta_{1}2^{k}}(x)}\underline{w}\right|<\varepsilon.
\]
Then 
\[
w(x)-\inf_{B_{3\delta_{1}2^{k}}(x)}\underline{w}\le C+CW^{\mu}_{1,n}(x,\delta_{1}2^{k+1}).
\]
Moreover, since 
\[
\inf_{A_{\delta2^{k-1},\delta2^{k}}}w=\inf_{A_{\delta2^{k-1},\delta2^{k}}}\underline{w}\,\,{\rm and}\,\,\left|\inf_{B_{3\delta_{1}2^{k}}(x)}\underline{w}-\inf_{A_{\delta2^{k-1},\delta2^{k}}}\underline{w}\right|\le C
\]
we have 
\[
\left|\inf_{B_{3\delta_{1}2^{k}}(x)}\underline{w}-\inf_{A_{\delta2^{k-1},\delta2^{k}}}w\right|\le C.
\]
Then
\begin{align*}
\int_{B_{r_{i}}(x_{i})}e^{nw}dx & \le C\int_{B_{r_{i}}(x_{i})}\exp(n\cdot\inf_{A_{\delta2^{k-1},\delta2^{k}}}w+C+n\cdot W^{\mu}_{1,n}(x,\delta_{1}2^{k+1}))dx\\
 & \le C\exp(n\cdot\inf_{A_{\delta2^{k-1},\delta2^{k}}}w)\int_{B_{r_{i}}(x_{i})}\exp(n\cdot W^{\mu}_{1,n}(x,\delta_{1}2^{k+1}))dx.
\end{align*}
Notice that for $\delta_{1}$ small, 
\begin{align*}
W^{\mu}_{1,n}(x,\delta_{1}2^{k+1}) & \le\int^{3r_{i}}_{0}+\int^{2^{k-3}}_{3r_{i}}\left(\mu(B(x,t))\right)^{\frac{1}{n-1}}\frac{dt}{t}\\
 & \le\mu(B(x,2^{k-3}))^{\frac{1}{n-1}}\log\frac{2^{k-3}}{3r_{i}}+W^{\mu}_{1,n}(x,3r_{i}).
\end{align*}
From \cite[Proposition 4.1]{MQ2021} or \cite[Lemma 2.2]{MW2025},
we know 
\begin{align*}
\int_{B_{3r_{i}}(x_{i})}\exp(n\cdot W^{\mu}_{1,n}(x,2^{k-3}))dx & \le\left(\frac{2^{k}}{3r_{i}}\right)^{n\mu(B(x_{i},2^{k-2}))^{\frac{1}{n-1}}}\int_{B_{3r_{i}}(x_{i})}\exp\left(n\cdot W^{\mu}_{1,n}(x,3r_{i})\right)dx\\
 & \le\left(\frac{2^{k}}{3r_{i}}\right)^{n\mu(B(x_{i},2^{k-2}))^{\frac{1}{n-1}}}Cr^{n}_{i}\\
 & \le2^{k\cdot n\mu(B(x_{i},2^{k-2}))^{\frac{1}{n-1}}}Cr^{n-n\mu(B(x_{i},2^{k-2}))^{\frac{1}{n-1}}}_{i}.
\end{align*}
Noticing that we may assume each $B_{i}$ intersects with at most
two $A_{\delta2^{k-1},\delta2^{k}}$. 
\begin{align*}
 & \frac{\int_{A_{\delta2^{k-1},\delta2^{k}}\cap E}e^{nw}dx}{\int_{A_{\delta2^{k-1},\delta2^{k}}}e^{n\underline{w}}dx}\\
\le & \frac{\sum_{B_{i}\cap A_{\delta2^{k-1},\delta2^{k}}\ne\emptyset}\int_{B_{r_{i}}(x_{i})}e^{nw}dx}{\int_{A_{\delta2^{k-1},\delta2^{k}}}e^{n\underline{w}}dx}\\
\le & \frac{C\exp(n\cdot\inf_{A_{\delta2^{k-1},\delta2^{k}}}w)\sum2^{k\cdot n\mu(B(x_{i},2^{k-2}))^{\frac{1}{n-1}}}r^{n-n\mu(B(x_{i},2^{k-2}))^{\frac{1}{n-1}}}_{i}}{\int_{A_{\delta2^{k-1},\delta2^{k}}}e^{n\underline{w}}dx}\\
\le & \frac{C}{|A_{\delta2^{k-1},\delta2^{k}}|}\sum2^{k\cdot n\mu(B(x_{i},2^{k-2}))^{\frac{1}{n-1}}}r^{n-n\mu(B(x_{i},2^{k-2}))^{\frac{1}{n-1}}}_{i}\\
\le & C\cdot\sum2^{k\left(n\mu(B(x_{i},2^{k-2}))^{\frac{1}{n-1}}-n\right)}r^{n-n\mu(B(x_{i},2^{k-2}))^{\frac{1}{n-1}}}_{i}\\
= & C\sum_{B_{i}\cap A_{\delta2^{k-1},\delta2^{k}}\ne\emptyset}\left(\frac{r_{i}}{|x_{i}|}\right)^{n-n\mu(B(x_{i},2^{k-2}))^{\frac{1}{n-1}}}.
\end{align*}
At last since $\mu(B(x_{i},2^{k-2}))^{\frac{1}{n-1}}$ is arbitrarily
small and from (\ref{Strong epsilon set}) we have
\[
\sum_{B_{i}\cap A_{\delta2^{k-1},\delta2^{k}}\ne\emptyset}\left(\frac{r_{i}}{|x_{i}|}\right)^{n-n\mu(B(x_{i},2^{k-2}))^{\frac{1}{n-1}}}
\]
 is arbitrarily small, for $k$ large. Then we know, uniformly for
$\delta\in(\frac{1}{2},1]$, 
\[
\lim_{k\to\infty}\frac{\int_{A_{\delta2^{k-1},\delta2^{k}}\cap E}e^{nw}dx}{\int_{A_{\delta2^{k-1},\delta2^{k}}}e^{n\underline{w}}dx}=0.
\]
 Then since 
\[
\int_{\mathbb{R}^{n}}e^{n\underline{w}}dx=\infty,
\]
 by Stolz's theorem, one easily get 
\begin{equation}
\lim_{l\to\infty}\frac{\int_{B^{\underline{g}}_{l}(0)\cap E}e^{nw}dx}{\int_{B^{\underline{g}}_{l}(0)}e^{n\underline{w}}dx}=0.\label{thit set volume ratio limit 0}
\end{equation}
Then combine (\ref{volume ratio limit less than 1}) and (\ref{thit set volume ratio limit 0}),
we have proved (\ref{volume ratio go to 1}). Then we know 
\[
\lim_{l\to\infty}\frac{\int_{B^{g}_{l}(0)}e^{nw}dx}{\omega_{n}l^{n}}\le\nu.
\]

Then we  prove the opposite side inequality. We  prove, uniformly
in $\overrightarrow{e}\in\mathbb{S}^{n-1}_{1}(0)\subset\mathbb{R}^{n}$,
\begin{equation}
\lim_{s\to+\infty}\frac{\int^{s}_{0}e^{w(t\overrightarrow{e})}dt}{\int^{s}_{0}e^{\underline{w}(t\overrightarrow{e})}dt}=1.\label{length ratio go to 1}
\end{equation}
It is obvious that the limit $\ge1$. For $\delta\in(\frac{1}{2},1]$,
we consider
\[
B_{r_{i}}(x_{i})\cap\{t\overrightarrow{e};t\in[\delta2^{k-1},\delta2^{k}]\}=\{t\overrightarrow{e};t\in(\alpha_{i},\beta_{i}),i=1,\cdots,n_{k}\}.
\]
Then, 
\[
\sum^{n_{k}}_{i=1}(\beta_{i}-\alpha_{i})\ll2^{k}.
\]

Consider 
\begin{align*}
\frac{\int^{\delta2^{k}}_{\delta2^{k-1}}e^{w(t\overrightarrow{e})}dt}{\int^{\delta2^{k}}_{\delta2^{k-1}}e^{\underline{w}(t\overrightarrow{e})}dt} & \le\frac{\int_{[\delta2^{k-1},\delta2^{k}]\backslash\cup_{i}[\alpha_{i},\beta_{i}]}e^{w(t\overrightarrow{e})}dt+\sum_{i}\int^{\beta_{i}}_{\alpha_{i}}e^{w(t\overrightarrow{e})}dt}{\int^{\delta2^{k}}_{\delta2^{k-1}}e^{\underline{w}(t\overrightarrow{e})}dt}.
\end{align*}
For the first term on the right, since 
\[
w(x)=\underline{w}(|x|)+o(1),\,\,x\to+\infty,x\notin\bigcup_{i}B^{e}_{r_{i}}(x_{i}).
\]
So for given $\varepsilon>0$, there exists $N>0$ such that if $k>N$,
\[
\frac{\int_{[\delta2^{k-1},\delta2^{k}]\backslash\cup_{i}[\alpha_{i},\beta_{i}]}e^{w(t\overrightarrow{e})}dt}{\int^{\delta2^{k}}_{\delta2^{k-1}}e^{\underline{w}(t\overrightarrow{e})}dt}\le\frac{\int^{\delta2^{k}}_{\delta2^{k-1}}e^{\underline{w}(t\overrightarrow{e})+\varepsilon}dt}{\int^{\delta2^{k}}_{\delta2^{k-1}}e^{\underline{w}(t\overrightarrow{e})}dt}\le e^{\varepsilon}.
\]
So we know that uniformly for $\delta\in(\frac{1}{2},1]$
\[
\limsup_{k\to+\infty}\frac{\int_{[\delta2^{k-1},\delta2^{k}]\backslash\cup_{i}[\alpha_{i},\beta_{i}]}e^{w(t\overrightarrow{e})}dt}{\int^{\delta2^{k}}_{\delta2^{k-1}}e^{\underline{w}(t\overrightarrow{e})}dt}\le1.
\]
For the second term 
\[
\sum_{i}\frac{\int^{\beta_{i}}_{\alpha_{i}}e^{w(t\overrightarrow{e})}dt}{\int^{\delta2^{k}}_{\delta2^{k-1}}e^{\underline{w}(t\overrightarrow{e})}dt}
\]

We notice that 
\begin{align*}
 & \int^{\beta_{i}}_{\alpha_{i}}e^{w(t\overrightarrow{e})}dt\\
= & \int_{B_{r_{i}}(x_{i})\cap\{t\overrightarrow{e};\delta2^{k-1}\le t\le\delta2^{k}\}}e^{w(t\overrightarrow{e})}dt\\
\le & \int_{B_{r_{i}}(x_{i})\cap\{t\overrightarrow{e};\delta2^{k-1}\le t\le\delta2^{k}\}}\exp(\inf_{A_{2^{k-1},2^{k}}}w+C+W^{\mu}_{1,n}(t\overrightarrow{e},\delta_{1}2^{k+1}))dt\\
\le & C\exp\left(\inf_{A_{\delta2^{k-1},\delta2^{k}}}w\right)\int_{B_{r_{i}}(x_{i})\cap\{t\overrightarrow{e};\delta2^{k-1}\le t\le\delta2^{k}\}}\exp(W^{\mu}_{1,n}(t\overrightarrow{e},\delta_{1}2^{k+1}))dt\\
\le & C\exp\left(\inf_{A_{\delta2^{k-1},\delta2^{k}}}w\right)\int_{B_{r_{i}}(x_{i})\cap\{t\overrightarrow{e};\delta2^{k-1}\le t\le\delta2^{k}\}}\exp\left(\mu(B(t\overrightarrow{e},2^{k-3}))^{\frac{1}{n-1}}\log\frac{2^{k-3}}{3r_{i}}+W^{\mu}_{1,n}(t\overrightarrow{e},3r_{i})\right)dt\\
\le & C\exp\left(\inf_{A_{\delta2^{k-1},\delta2^{k}}}w\right)\left(\frac{2^{k-3}}{3r_{i}}\right)^{\mu(\mathbb{R}^{n}\backslash B(0,2^{k-3}))^{\frac{1}{n-1}}}\int^{\beta_{i}}_{\alpha_{i}}\exp\left(W^{\mu}_{1,n}(t\overrightarrow{e},3r_{i})\right)dt
\end{align*}

From Lemma \ref{integral along a line} below, we have 
\begin{align*}
 & \int^{\beta_{i}}_{\alpha_{i}}\exp\left(W^{\mu}_{1,n}(t\overrightarrow{e},3r_{i})\right)dt\\
\le & \beta_{i}-\alpha_{i}+\frac{2^{\frac{1}{n-1}}}{\frac{1}{n-1}-\mu(A_{\delta2^{k-2},\delta2^{k+1}})^{\frac{1}{n-1}}}(3r_{i})^{\frac{1}{n-1}}(\beta_{i}-\alpha_{i})^{\frac{n-2}{n-1}}\mu(A_{\delta2^{k-2},\delta2^{k+1}})^{\frac{1}{n-1}}\\
\le & Cr_{i}.
\end{align*}

So 
\begin{align*}
\int^{\beta_{i}}_{\alpha_{i}}e^{w(t\overrightarrow{e})}dt & \le C\exp\left(\inf_{A_{\delta2^{k-1},\delta2^{k}}}w\right)\left(\frac{2^{k-3}}{3r_{i}}\right)^{\mu(\mathbb{R}^{n}\backslash B(0,2^{k-3}))^{\frac{1}{n-1}}}r_{i}.
\end{align*}
Then 
\begin{align*}
\sum_{i}\frac{\int^{\beta_{i}}_{\alpha_{i}}e^{w(t\overrightarrow{e})}dt}{\int^{\delta2^{k}}_{\delta2^{k-1}}e^{\underline{w}(t\overrightarrow{e})}dt} & \le C\sum_{B_{r_{i}}(x_{i})\cap\{t\overrightarrow{e};t\in[\delta2^{k-1},\delta2^{k}]\}\ne\emptyset}\frac{\exp\left(\inf_{A_{\delta2^{k-1},\delta2^{k}}}w\right)\left(\frac{2^{k-3}}{3r_{i}}\right)^{\mu(\mathbb{R}^{n}\backslash B(0,2^{k-3}))^{\frac{1}{n-1}}}r_{i}}{\int^{\delta2^{k}}_{\delta2^{k-1}}e^{\underline{w}(t\overrightarrow{e})}dt}\\
 & =C\sum_{B_{r_{i}}(x_{i})\cap\{t\overrightarrow{e};t\in[\delta2^{k-1},\delta2^{k}]\}\ne\emptyset}\frac{\left(\frac{2^{k-3}}{3r_{i}}\right)^{\mu(\mathbb{R}^{n}\backslash B(0,2^{k-3}))^{\frac{1}{n-1}}}r_{i}}{\int^{\delta2^{k}}_{\delta2^{k-1}}\exp\left(\underline{w}(t\overrightarrow{e})-\inf_{A_{\delta2^{k-1},\delta2^{k}}}w\right)dt}\\
 & \le C\sum_{B_{r_{i}}(x_{i})\cap\{t\overrightarrow{e};t\in[\delta2^{k-1},\delta2^{k}]\}\ne\emptyset}\frac{\left(\frac{2^{k-3}}{3r_{i}}\right)^{\mu(\mathbb{R}^{n}\backslash B(0,2^{k-3}))^{\frac{1}{n-1}}}r_{i}}{2^{k}}\\
 & \le C\sum_{B_{r_{i}}(x_{i})\cap\{t\overrightarrow{e};t\in[\delta2^{k-1},\delta2^{k}]\}\ne\emptyset}\left(\frac{r_{i}}{2^{k}}\right)^{1-\mu(\mathbb{R}^{n}\backslash B(0,2^{k-3}))^{\frac{1}{n-1}}}
\end{align*}
Again from (\ref{Strong epsilon set}), we know 
\[
\sum_{B_{r_{i}}(x_{i})\cap\{t\overrightarrow{e};t\in[\delta2^{k-1},\delta2^{k}]\}\ne\emptyset}\left(\frac{r_{i}}{2^{k}}\right)^{1-\mu(\mathbb{R}^{n}\backslash B(0,2^{k-3}))^{\frac{1}{n-1}}}
\]
is arbitrarily small and uniformly in $\overrightarrow{e}$. 

Then uniformly in $\overrightarrow{e}$
\[
\lim_{k\to\infty}\frac{\int^{\delta2^{k}}_{\delta2^{k-1}}e^{w(t\overrightarrow{e})}dt}{\int^{\delta2^{k}}_{\delta2^{k-1}}e^{\underline{w}(t\overrightarrow{e})}dt}=1
\]
and hence by Stolz's theorem again, (\ref{length ratio go to 1})
holds. 

Then from (\ref{length ratio go to 1}), 
\begin{equation}
B^{g}_{l}(0)\supset B^{\underline{g}}_{l-\varepsilon(l)}(0)\label{B^g contains B^=00005Cunderlineg_g}
\end{equation}
 for some $0<\varepsilon(l)=o(l)$. Then 
\[
\lim_{l\to\infty}\frac{V_{g}(B^{g}_{l}(0))}{\omega_{n}l^{n}}\ge\lim_{l\to\infty}\frac{V_{\underline{g}}(B^{\underline{g}}_{l-\varepsilon(l)}(0))}{\omega_{n}l^{n}}=\nu.
\]
 Then we proved the theorem.

\end{proof}

\begin{lem}\label{integral along a line}

Suppose that $\overrightarrow{e}$ is a unit vector, and $[\alpha,\beta]\subset\mathbb{R}^{+},D>1$,
and $\alpha>2D$. We let $\Omega$ be the $D$ neighborhood of $\{t\overrightarrow{e};t\in[\alpha,\beta]\}$
in $\mathbb{R}^{n}$. Assume $\mu(\Omega)^{\frac{1}{n-1}}<\frac{1}{n-1}$.
Then 
\[
\int^{\beta}_{\alpha}\exp\left(W^{\mu}_{1,n}\left(t\overrightarrow{e},D\right)\right)dt\le\beta-\alpha+\frac{2^{\frac{1}{n-1}}}{\frac{1}{n-1}-\mu(\Omega)^{\frac{1}{n-1}}}D^{\frac{1}{n-1}}(\beta-\alpha)^{\frac{n-2}{n-1}}\mu(\Omega)^{\frac{1}{n-1}}.
\]

\end{lem}

\begin{proof}

We use $P$ to denote the nearest point projection from $\Omega$
to $\{t\overrightarrow{e};t\in[\alpha-D,\beta+D]\}$. Then $\mu\lfloor\Omega$
is projected to be a measure $\hat{\mu}$ which is supported on $\{t\overrightarrow{e};t\in[\alpha-D,\beta+D]\}$.
It is easy to see that, $\forall t\in[\alpha,\beta]$, $\mu(B(t\overrightarrow{e},s))\le\hat{\mu}(B(t\overrightarrow{e},s)),\forall s\in[0,D]$. 

Then for $t\in[\alpha,\beta]$ a.e.
\begin{align*}
W^{\mu}_{1,n}\left(t\overrightarrow{e},D\right) & =\int^{D}_{0}\mu(B(t\overrightarrow{e},s))^{\frac{1}{n-1}}\frac{ds}{s}\le\int^{D}_{0}\hat{\mu}(B(t\overrightarrow{e},s))^{\frac{1}{n-1}}\frac{ds}{s}\\
 & =\hat{\mu}(B(t\overrightarrow{e},s))^{\frac{1}{n-1}}\log s|^{D}_{0}+\int^{D}_{0}\log\frac{1}{s}d\hat{\mu}(B(t\overrightarrow{e},s))^{\frac{1}{n-1}}.
\end{align*}
Since $\hat{\mu}(B(t\overrightarrow{e},s))\le C(t)s$ a.e., we have
\[
W^{\mu}_{1,n}(t\overrightarrow{e},D)\le\hat{\mu}(B(t\overrightarrow{e},D))^{\frac{1}{n-1}}\log D+\int^{D}_{0}\log\frac{1}{s}d\hat{\mu}(B(t\overrightarrow{e},s))^{\frac{1}{n-1}}.
\]
We denote $\gamma_{t}=\hat{\mu}(B(t\overrightarrow{e},D))^{\frac{1}{n-1}}.$
Then using Jensen's formula we get 
\begin{align*}
\exp\left(W^{\mu}_{1,n}(t\overrightarrow{e},D)\right) & \le D^{\gamma_{t}}\int^{D}_{0}\frac{1}{s^{\gamma_{t}}}\frac{1}{\gamma_{t}}d\hat{\mu}(B(t\overrightarrow{e},s))^{\frac{1}{n-1}}\\
 & =D^{\gamma_{t}}\left(\frac{1}{s^{\gamma_{t}}}\frac{1}{\gamma_{t}}\hat{\mu}(B(t\overrightarrow{e},s))^{\frac{1}{n-1}}|^{D}_{0}+\int^{D}_{0}\hat{\mu}(B(t\overrightarrow{e},s))^{\frac{1}{n-1}}\frac{1}{s^{\gamma_{t}+1}}ds\right).
\end{align*}
Since $\gamma_{t}\le\mu(\Omega)^{\frac{1}{n-1}}<\frac{1}{n-1}$, the
first term on the right 
\[
D^{\gamma_{t}}\frac{1}{s^{\gamma_{t}}}\frac{1}{\gamma_{t}}\hat{\mu}(B(t\overrightarrow{e},s))^{\frac{1}{n-1}}|^{D}_{0}=1,\,\,a.e..
\]
Then 
\[
\exp\left(W^{\mu}_{1,n}(t\overrightarrow{e},D)\right)\le1+D^{\gamma_{t}}\int^{D}_{0}\hat{\mu}(B(t\overrightarrow{e},s))^{\frac{1}{n-1}}\frac{1}{s^{\gamma_{t}+1}}ds.
\]
We let $\gamma=\mu(\Omega)^{\frac{1}{n-1}}\ge\gamma_{t}$. Then
\begin{align*}
\exp\left(W^{\mu}_{1,n}(t\overrightarrow{e},D)\right) & \le1+D^{-1}\int^{D}_{0}\hat{\mu}(B(t\overrightarrow{e},s))^{\frac{1}{n-1}}\left(\frac{D}{s}\right)^{\gamma_{t}+1}ds\\
 & \le1+D^{-1}\int^{D}_{0}\hat{\mu}(B(t\overrightarrow{e},s))^{\frac{1}{n-1}}\left(\frac{D}{s}\right)^{\gamma+1}ds\\
 & =1+D^{\gamma}\int^{D}_{0}\hat{\mu}(B(t\overrightarrow{e},s))^{\frac{1}{n-1}}\left(\frac{1}{s}\right)^{\gamma+1}ds.
\end{align*}

Then
\begin{align*}
\int^{\beta}_{\alpha}\exp\left(W^{\mu}_{1,n}(t\overrightarrow{e},D)\right)dt & \le(\beta-\alpha)+\int^{\beta}_{\alpha}D^{\gamma}\int^{D}_{0}\hat{\mu}(B(t\overrightarrow{e},s))^{\frac{1}{n-1}}\frac{1}{s^{\gamma+1}}dsdt\\
 & =\beta-\alpha+\int^{D}_{0}D^{\gamma}\frac{1}{s^{\gamma+1}}ds\int^{\beta}_{\alpha}\hat{\mu}(B(t\overrightarrow{e},s))^{\frac{1}{n-1}}dt\\
 & =\beta-\alpha+D^{\gamma}\int^{D}_{0}\frac{1}{s^{\gamma+1}}ds\left(\int^{\beta}_{\alpha}\hat{\mu}(B(t\overrightarrow{e},s))dt\right)^{\frac{1}{n-1}}(\beta-\alpha)^{\frac{n-2}{n-1}}\\
 & \le\beta-\alpha+D^{\gamma}(\beta-\alpha)^{\frac{n-2}{n-1}}\int^{D}_{0}\frac{(2s)^{\frac{1}{n-1}}}{s^{\gamma+1}}ds\cdot\mu(\Omega)^{\frac{1}{n-1}}\\
 & =\beta-\alpha+\frac{2^{\frac{1}{n-1}}}{\frac{1}{n-1}-\gamma}D^{\frac{1}{n-1}}(\beta-\alpha)^{\frac{n-2}{n-1}}\mu(\Omega)^{\frac{1}{n-1}}.
\end{align*}

\end{proof}

\section{Proof of Theorem \ref{main thm 3} and \ref{main thm 4 generalization of Chen-Zhu's theorem}}

We suppose that $(M^{n},g)$ is a complete noncompact locally conformally
flat manifold with nonnegative Ricci curvature. The classifications
of such manifolds are studied in \cite{Zhu1994,CH2006}. We prove
Theorem \ref{main thm 3} according to the following classification
theorem of Carron and Herzlich in \cite{CH2006}. 

\begin{thm}\label{Carron-Herzlich classification thm}\cite[Theorem A]{CH2006}
Let $(M,g)$ be a complete conformally flat manifold of dimension
$n\ge3$ with $Ric\ge0$. Then exactly one of the following holds:
\begin{enumerate}
\item $M$ is globally conformally equivalent to $\mathbb{R}^{n}$ with
a conformal non-flat metric with $Ric\ge0$;
\item $M$ is globally conformally equivalent to a spaceform of positive
curvature, endowed with a conformal metric with $Ric\ge0$;
\item $M$ is locally isometric to the cylinder $\mathbb{R}\times\mathbb{S}^{n-1}_{a}$,
where $\mathbb{S}^{n-1}_{a}$ is the round sphere of radius $a$;
\item $M$ is isometric to a complete flat manifold.
\end{enumerate}
\end{thm}

Case 2 does not happen since the manifolds involved are compact.

If Case 3 happens, then 
\[
M=(\mathbb{R}\times\mathbb{S}^{n-1}_{a})/\Gamma,\Gamma\cong\pi_{1}(M).
\]
 Since $M$ is non-compact, the projection from $\Gamma$ to ${\rm Iso}(\mathbb{R})$
contains no nontrivial translations, otherwise the quotient would
be compact. Thus this projection has order at most two, while its
kernel is finite because it acts properly discontinuously on $\mathbb{S}^{n-1}_{a}$.
Therefore $M$ is a finite, possibly twisted, quotient of the cylinder,
and 
\[
V_{g}(B^{g}_{l}(p))=\frac{2a^{n-1}|\mathbb{S}^{n-1}|}{|\Gamma|}l+O(1),\,\int_{B^{g}_{l}(p)}R_{g}d\mu_{g}=\frac{2(n-1)(n-2)a^{n-3}|\mathbb{S}^{n-1}|}{|\Gamma|}l+O(1).
\]
Consequently $\nu=0$; if $n>3$, then 
\[
\lim_{l\to\infty}l^{2-n}\int_{B^{g}_{l}(p)}R_{g}d\mu_{g}=0=(n-1)(\nu^{\frac{n-3}{n-1}}-\nu)|\mathbb{S}^{n-1}|
\]
whereas if $n=3,$ then 
\[
\lim_{l\to\infty}l^{-1}\int_{B^{g}_{l}(p)}R_{g}d\mu_{g}=\frac{16\pi}{|\Gamma|}=\frac{16\pi}{|\pi_{1}(M)|}.
\]

If Case 4 happens, then trivially 
\[
\int_{B^{g}_{l}(p)}R_{g}=0.
\]
If $M$ is flat $\mathbb{R}^{n},$ $\nu=1$. Otherwise, $\nu=0$.
Then if $n>3$, there holds 
\[
\frac{1}{l^{n-2}}\int_{B^{g}_{l}(p)}R_{g}=0=(n-1)(\nu^{\frac{n-3}{n-1}}-\nu)|\mathbb{S}^{n-1}|.
\]

Then we only need to deal with Case 1 (the main case), i.e. $(M,g=e^{2w(x)}|dx|^{2})$
with $Ric_{g}\ge0$ and $g$ is not flat. We have 
\[
-\Delta_{n}w\ge0,\,\,x\in\mathbb{R}^{n}.
\]
 Then we have that 
\[
w(x)=\underline{w}(|x|)+o(1),x\in\mathbb{R}^{n}\backslash E
\]
for an $\mathcal{E}$ set $E$. We consider $\underline{g}=e^{2\underline{w}(|x|)}g_{\mathbb{R}^{n}}.$
By \cite[Theorem 1.3]{MQ2021}, we know 
\[
-\liminf_{x\to\infty}\frac{w(x)}{\log|x|}=\lim_{x\to\infty}\frac{\underline{w}(x)}{\log\frac{1}{|x|}}=\tilde{m}\in(0,1].
\]
We first show that the scalar curvature $\underline{R}$ of $\underline{g}$,
defined as 
\[
\underline{R}=-\frac{4(n-1)}{n-2}\underline{v}^{-\frac{n+2}{n-2}}\Delta\underline{v},{\rm where}\,\underline{v}=e^{\frac{n-2}{2}\underline{w}}
\]
and understood as a Radon measure, satisfies the conclusion of Theorem
\ref{main thm 3}. 

We write $\underline{g}$ as
\begin{align*}
\underline{g} & =dl^{2}+\psi(l)^{2}g_{\mathbb{S}^{n-1}}\\
 & =dl^{2}+\theta(l)^{2}l^{2}g_{\mathbb{S}^{n-1}}.
\end{align*}

\begin{lem}\label{psi theta information}$\psi(l)$ is concave, monotonically
nondecreasing in $l\in[0,+\infty)$. $0\le\psi'(l\pm0)\le1.$ $\psi''(l)$
is a nonpositive Radon measure. 

\begin{align}
\underline{R}(l) & =(n-1)\left((n-2)\frac{1-\psi'(l)^{2}}{\psi(l)^{2}}-2\frac{\psi''(l)}{\psi(l)}\right)\nonumber \\
 & =\frac{(n-1)\left((n-2)\left(1-l^{2}\theta'(l)^{2}-\theta(l)^{2}\right)-2l\theta(l)\left(n\theta'(l)+l\theta''(l)\right)\right)}{l^{2}\theta(l)^{2}}.\label{scalar curvature expression}
\end{align}
$\theta(l)$ has the regularity as a concave function does and is
monotonically decreasing in $l$. Moreover, 
\begin{align*}
\theta(0)\,\,\, & \le1,\\
\theta(\infty) & =\lim_{l\to\infty}\theta(l)=1-\tilde{m},\\
\lim_{l\to\infty}\,\,\, & l\theta'(l)=0.
\end{align*}

\end{lem}

\begin{proof}Since $\underline{w}(|x|)$ is $n$-superharmonic in
$\mathbb{R}^{n}$, we easily verify, in the sense of distribution,
\[
w''(r)+\frac{1}{r}w'(r)\le0.
\]
Then
\[
\psi''(l)=e^{-w(r)}r\left(w''+\frac{1}{r}w'\right)\le0.
\]
The expression (\ref{scalar curvature expression}) is straightforward
to check. 

Again easy calculation yields, in the sense of distribution, 
\begin{align}
\theta'(l)l+\theta(l) & =\underline{w}'(r)r+1,\label{=00005Ctheta'(l)l+=00005Ctheta(l)}\\
\theta''(l)l+2\theta'(l) & =e^{-w(r)}(\underline{w}'(r)+\underline{w}''(r)r).\label{=00005Ctheta''(l)+2=00005Ctheta'(l)}
\end{align}
So from (\ref{=00005Ctheta''(l)+2=00005Ctheta'(l)}), we have
\[
(\theta'(l)l^{2})'\le0.
\]
Since from (\ref{=00005Ctheta'(l)l+=00005Ctheta(l)})
\[
\liminf_{l\to0}\theta'(l)l^{2}=\liminf_{l\to0}(\underline{w}'(r)r+1-\theta(l))l=0,
\]
 we know $\theta'(l)\le0$. 
\begin{align*}
\theta(\infty)=\lim_{l\to\infty}\theta(l) & =\lim_{l\to\infty}\frac{e^{\underline{w}(r)}r}{l}=\lim_{r\to\infty}\left(e^{\underline{w}(r)}r'(l)+e^{\underline{w}(r)}r\underline{w}'(r)r'(l)\right)\\
 & =\lim_{r\to\infty}(1+rw'(r))\\
 & =1-\tilde{m},
\end{align*}
where the last equality follows from Lemma \ref{radial symmetric n-superharmonic functions}.
Then from (\ref{=00005Ctheta'(l)l+=00005Ctheta(l)}), 
\[
\lim_{l\to\infty}l\theta'(l)=0.
\]
And $0\le\psi'(l)=\theta'(l)l+\theta(l)\le\theta(0)\le1$. 

\end{proof}

We then have
\begin{align}
 & \int_{B^{\underline{g}}_{l}(0)}\underline{R}d\mu_{\underline{g}}\nonumber \\
= & (n-1)|\mathbb{S}^{n-1}|\int^{l}_{0}\left((n-2)\left(1-s^{2}\theta'(s)^{2}\right)-2s\theta(s)\left(n\theta'(s)+s\theta''(s)\right)-(n-2)\theta(s)^{2}\right)s^{n-3}\theta(s)^{n-3}ds.\nonumber \\
= & (n-1)|\mathbb{S}^{n-1}|(\int^{l}_{0}\left((n-2)\left(1-\theta(s)^{2}\right)-2s\theta(s)\theta'(s)+(n-2)s^{2}\theta'(s)^{2}\right)s^{n-3}\theta(s)^{n-3}ds\nonumber \\
 & -2l^{n-1}\theta(l)^{n-2}\theta'(l)).\label{=00005CunderlineR's expression}
\end{align}
 From Lemma \ref{psi theta information}, $l\theta(l)^{n-2}\theta'(l)\to0$
as $l\to\infty$. Then
\begin{align*}
 & \lim_{l\to+\infty}l^{2-n}\int_{B^{\underline{g}}_{l}(0)}\underline{R}d\mu_{\underline{g}}\\
= & \lim_{l\to+\infty}\frac{(n-1)|\mathbb{S}^{n-1}|\left((n-2)\left(1-\theta(l)^{2}\right)-2l\theta(l)\theta'(l)+(n-2)l^{2}\theta'(l)^{2}\right)l^{n-3}\theta(l)^{n-3}}{(n-2)l^{n-3}}\\
= & \lim_{l\to+\infty}\frac{(n-1)|\mathbb{S}^{n-1}|\left((n-2)\left(1-\theta(l)^{2}\right)\right)\theta(l)^{n-3}}{(n-2)}\\
= & \begin{cases}
(n-1)|\mathbb{S}^{n-1}|(1-\theta(\infty)^{2})\theta(\infty)^{n-3} & n>3,\\
8\pi(1-\theta(\infty)^{2}) & n=3,
\end{cases}\\
= & \begin{cases}
(n-1)|\mathbb{S}^{n-1}|(\nu^{\frac{n-3}{n-1}}-\nu) & n>3,\\
8\pi(1-\nu) & n=3.
\end{cases}
\end{align*}
Then we estimate the difference between $\int_{B^{g}_{l}(0)}Rd\mu_{g}$
and $\int_{B^{\underline{g}}_{l}(0)}\underline{R}d\mu_{\underline{g}}$
. 
\[
\int_{B^{g}_{l}}Rd\mu_{g}\le\int_{B^{\underline{g}}_{l}}R\mu_{g}.
\]
We let $e^{2w}g_{edu}=v^{\frac{4}{n-2}}g_{edu}$. Assume 
\[
l=\int^{r}_{0}e^{\underline{w}(t)}dt.
\]
Then 
\begin{align*}
\int_{B^{\underline{g}}_{l}}Rd\mu_{g} & =\frac{4(n-1)}{n-2}\left(\int_{B_{r}(0)}|\nabla v|^{2}-\int_{\partial B_{r}(0)}\frac{\partial v}{\partial\nu}vdS(\partial B_{r})\right),
\end{align*}
where $\nu$ is the unit outward normal of $\partial B_{r}(0)$. 

\begin{lem}\label{gradient ratio go to 1}
\[
\lim_{r\to\infty}\frac{\int_{B_{r}(0)}|\nabla v|^{2}}{\int_{B_{r}(0)}|\nabla\underline{v}|^{2}}=1.
\]

\end{lem}

\begin{proof}

It is equivalent to prove that 
\[
\lim_{r\to\infty}\frac{\int_{B_{r}(0)}e^{(n-2)w}|\nabla w|^{2}dx}{\int_{B_{r}(0)}e^{(n-2)\underline{w}}|\nabla\underline{w}|^{2}dx}=1.
\]
 From \cite[Theorem 1.3]{MQ2021} $\underline{w}(r)\ge\tilde{m}\log\frac{1}{r}-C,\,\,\tilde{m}\in(0,1]$,
and from Lemma \ref{radial symmetric n-superharmonic functions},
we know 
\[
\lim_{r\to\infty}r|\nabla\underline{w}|=\tilde{m}.
\]
Then 
\[
e^{(n-2)\underline{w}}|\nabla\underline{w}|^{2}\ge\frac{C\tilde{m}^{2}}{r^{\tilde{m}(n-2)+2}}
\]
and hence as $r\to+\infty$, for some $r_{0}>0$
\begin{equation}
\int_{B_{r}(0)}e^{(n-2)\underline{w}}|\nabla\underline{w}|^{2}\ge C\tilde{m}^{2}\int^{r}_{r_{0}}s^{n-1-\tilde{m}(n-2)-2}ds=C\tilde{m}^{2}\int^{r}_{r_{0}}s^{(1-\tilde{m})(n-2)-1}ds\to+\infty.\label{=00005Cunderlinewintegralinfinity}
\end{equation}
 Then by using Stolz's theorem, it suffices to prove that, uniformly
for $\delta\in(\frac{1}{2},1]$,
\[
\lim_{k\to\infty}\frac{\int_{A_{\text{\ensuremath{\delta}}2^{k-1},\delta2^{k}}}e^{(n-2)w}|\nabla w|^{2}}{\int_{A_{\text{\ensuremath{\delta}}2^{k-1},\delta2^{k}}}e^{(n-2)\underline{w}}|\nabla\underline{w}|^{2}}=1.
\]
For given small $\lambda>0$, we let 
\[
E_{\lambda}=\{x\in A_{\text{\ensuremath{\delta}}2^{k-1},\delta2^{k}};\left||\nabla w|-|\nabla\underline{w}|\right|>\lambda2^{-k}\}
\]
and divide the numerator into three parts,
\begin{align*}
\int_{A_{\text{\ensuremath{\delta}}2^{k-1},\delta2^{k}}}e^{(n-2)w}|\nabla w|^{2} & =\int_{A_{\text{\ensuremath{\delta}}2^{k-1},\delta2^{k}}\cap E}+\int_{E_{\lambda}\backslash E}+\int_{A_{\text{\ensuremath{\delta}}2^{k-1},\delta2^{k}}\backslash(E\cup E_{\lambda})}e^{(n-2)w}|\nabla w|^{2}.
\end{align*}
\begin{align*}
 & \frac{\int_{A_{\text{\ensuremath{\delta}}2^{k-1},\delta2^{k}}\cap E}e^{(n-2)w}|\nabla w|^{2}}{\int_{A_{\text{\ensuremath{\delta}}2^{k-1},\delta2^{k}}}e^{(n-2)\underline{w}}|\nabla\underline{w}|^{2}}\\
\le & \frac{\sum_{B_{r_{i}}(x_{i})\cap A_{\text{\ensuremath{\delta}}2^{k-1},\delta2^{k}}\ne\emptyset}\int_{B_{r_{i}}(x_{i})}e^{(n-2)w}|\nabla w|^{2}}{\int_{A_{\text{\ensuremath{\delta}}2^{k-1},\delta2^{k}}}e^{(n-2)\underline{w}}|\nabla\underline{w}|^{2}}\\
\le & \frac{\sum\int_{B_{r_{i}}(x_{i})}\exp\left((n-2)\inf_{A_{\text{\ensuremath{\delta}}2^{k-1},\delta2^{k}}}w+C+CW^{\mu}_{1,n}(x,\delta_{1}2^{k+1}))\right)|\nabla w|^{2}}{\int_{A_{\text{\ensuremath{\delta}}2^{k-1},\delta2^{k}}}e^{(n-2)\underline{w}}|\nabla\underline{w}|^{2}}\\
\le & \frac{C\exp((n-2)\inf_{A_{\text{\ensuremath{\delta}}2^{k-1},\delta2^{k}}}w)\sum\int_{B_{r_{i}}(x_{i})}\exp(CW^{\mu}_{1,n}(x,\delta_{1}2^{k+1}))|\nabla w|^{2}}{\int_{A_{\text{\ensuremath{\delta}}2^{k-1},\delta2^{k}}}e^{(n-2)\underline{w}}|\nabla\underline{w}|^{2}}\\
\le & \frac{C\exp((n-2)\inf_{A_{\delta2^{k-1},\delta2^{k}}}w)\sum(\int_{B_{r_{i}}(x_{i})}\exp(CW^{\mu}_{1,n}(x,\delta_{1}2^{k+1})))^{\frac{n-2}{n+\varepsilon}}(\int_{B_{r_{i}}(x_{i})}|\nabla w|^{\frac{2(n+\varepsilon)}{2+\varepsilon}})^{\frac{2+\varepsilon}{n+\varepsilon}}}{\int_{A_{\text{\ensuremath{\delta}}2^{k-1},\delta2^{k}}}e^{(n-2)\underline{w}}|\nabla\underline{w}|^{2}}
\end{align*}

\begin{lem}
\[
\int_{B_{r_{i}}(x_{i})}|\nabla w|^{\frac{2(n+\varepsilon)}{2+\varepsilon}}\le Cr^{\frac{\varepsilon(n-2)}{2+\varepsilon}}_{i}.
\]

\end{lem}

\begin{proof}

We invoke \cite[Theorem 1.1]{DM2011}, i.e. for $x\in B_{r_{i}}(x_{i})$
\begin{align*}
|\nabla w(x)| & \le C\fint_{B_{2^{k-3}}(x)}|\nabla w|dx+CW^{\mu}_{\frac{1}{n},n}(x,2^{k-2})\\
 & \le2^{-k}C+CW^{\mu}_{\frac{1}{n},n}(x,2^{k-2}).
\end{align*}
Here since $w(x)$ is a smooth function which is a conformal factor,
we may verify the assumption of \cite[Theorem 1.1]{DM2011}. Then
\begin{align*}
\int_{B_{r_{i}}(x_{i})}|\nabla w|^{\frac{2(n+\varepsilon)}{2+\varepsilon}} & \le\int_{B_{r_{i}}(x_{i})}\left(2^{-k}C+CW^{\mu}_{\frac{1}{n},n}(x,2^{k-2})\right)^{\frac{2(n+\varepsilon)}{2+\varepsilon}}\\
 & \le C\int_{B_{r_{i}}(x_{i})}\left(2^{-k\frac{2(n+\varepsilon)}{2+\varepsilon}}+W^{\mu}_{\frac{1}{n},n}(x,2^{k-2})^{\frac{2(n+\varepsilon)}{2+\varepsilon}}\right)dx\\
 & \le Cr^{n}_{i}2^{-k\frac{2(n+\varepsilon)}{2+\varepsilon}}+C\int_{B_{r_{i}}(x_{i})}W^{\mu}_{\frac{1}{n},n}(x,2^{k-2})^{\frac{2(n+\varepsilon)}{2+\varepsilon}}dx.
\end{align*}
For the first term 
\begin{align*}
Cr^{n}_{i}2^{-k\frac{2(n+\varepsilon)}{2+\varepsilon}} & =Cr^{\frac{\varepsilon(n-2)}{2+\varepsilon}}_{i}\left(\frac{r_{i}}{2^{k}}\right)^{\frac{2(n+\varepsilon)}{2+\varepsilon}}\ll r^{\frac{\varepsilon(n-2)}{2+\varepsilon}}_{i}.
\end{align*}
For the second term, from \cite[(2.6.1),Proposition 5.1.4 (a), (6.3.6)]{AH1996},
we have 
\begin{align*}
 & \text{meas}\left(\{x\in B_{r_{i}}(x_{i});W^{\mu}_{\frac{1}{n},n}(x,2^{k-2})\ge\lambda\}\right)\\
\le & \min\left\{ \text{meas}(B_{r_{i}}(x_{i})),\frac{C\mu(A_{2^{k-2},2^{k+1}})^{\frac{n}{n-1}}}{\lambda^{n}}\right\} .
\end{align*}
 Then we can estimate 
\[
\int_{B_{r_{i}}(x_{i})}W^{\mu}_{\frac{1}{n},n}(x,2^{k-2})^{\frac{2(n+\varepsilon)}{2+\varepsilon}}dx\le Cr^{\frac{\varepsilon(n-2)}{2+\varepsilon}}_{i}.
\]

\end{proof}

Then we continue to prove Lemma \ref{gradient ratio go to 1}
\begin{align*}
 & \frac{\int_{A_{\text{\ensuremath{\delta}}2^{k-1},\delta2^{k}}\cap E}e^{(n-2)w}|\nabla w|^{2}}{\int_{A_{\text{\ensuremath{\delta}}2^{k-1},\delta2^{k}}}e^{(n-2)\underline{w}}|\nabla\underline{w}|^{2}}\\
\le & \frac{C\sum((\frac{2^{k}}{3r_{i}})^{C\mu(B(x_{i},2^{k-2}))^{\frac{1}{n-1}}}r^{n}_{i})^{\frac{n-2}{n+\varepsilon}}(r^{\frac{\varepsilon(n-2)}{n+\varepsilon}}_{i})}{\int_{A_{\text{\ensuremath{\delta}}2^{k-1},\delta2^{k}}}e^{(n-2)(\underline{w}-\inf_{A_{\text{\ensuremath{\delta}}2^{k-1},\delta2^{k}}}w)}|\nabla\underline{w}|^{2}}\\
\le & \frac{C\sum((\frac{2^{k}}{3r_{i}})^{C\mu(B(x_{i},2^{k-2}))^{\frac{1}{n-1}}}r^{n}_{i})^{\frac{n-2}{n+\varepsilon}}(r^{\frac{\varepsilon(n-2)}{n+\varepsilon}}_{i})}{\int_{A_{\text{\ensuremath{\delta}}2^{k-1},\delta2^{k}}}|\nabla\underline{w}|^{2}}\\
\le & \frac{C\sum((\frac{2^{k}}{3r_{i}})^{C\mu(B(x_{i},2^{k-2}))^{\frac{1}{n-1}}}r^{n}_{i})^{\frac{n-2}{n+\varepsilon}}(r^{\frac{\varepsilon(n-2)}{n+\varepsilon}}_{i})}{2^{(n-2)k}}\\
\le & C2^{k(C\mu(B(x_{i},2^{k-2}))^{\frac{1}{n-1}}\frac{n-2}{n+\varepsilon}-(n-2))}\sum r^{n-2+\frac{n-2}{n+\varepsilon}(-C\mu(B(x_{i},2^{k-2}))^{\frac{1}{n-1}})}_{i}\\
\le & C\sum_{B_{r_{i}}(x_{i})\cap A_{\text{\ensuremath{\delta}}2^{k-1},\delta2^{k}}\ne\emptyset}\left(\frac{r_{i}}{2^{k}}\right)^{n-2-C\mu(B(x_{i},2^{k-2}))^{\frac{1}{n-1}}},
\end{align*}
in which the third inequality follows since 
\[
\lim_{r\to\infty}r|\nabla\underline{w}|=\tilde{m}>0,
\]
which implies 
\[
\int_{A_{\text{\ensuremath{\delta}}2^{k-1},\delta2^{k}}}|\nabla\underline{w}|^{2}\ge C2^{(n-2)k}.
\]

From (\ref{Strong epsilon set}), for given $\lambda>0$, there exists
$k_{0}>0$, such that when $k>k_{0}$, for $\delta\in(\frac{1}{2},1]$,
\begin{align*}
 & \frac{\int_{A_{\text{\ensuremath{\delta}}2^{k-1},\delta2^{k}}\cap E}e^{(n-2)w}|\nabla w|^{2}}{\int_{A_{\text{\ensuremath{\delta}}2^{k-1},\delta2^{k}}}e^{(n-2)\underline{w}}|\nabla\underline{w}|^{2}}\\
\le & C\sum_{B_{r_{i}}(x_{i})\cap A_{\text{\ensuremath{\delta}}2^{k-1},\delta2^{k}}\ne\emptyset}\left(\frac{r_{i}}{2^{k}}\right)^{n-2-C\mu(B(x_{i},2^{k-2}))^{\frac{1}{n-1}}}\le\lambda.
\end{align*}

Now we let 
\begin{align*}
w_{\delta,k}(\xi) & =w(\delta2^{k}\xi)-\underline{w}(\delta2^{k}),\\
\underline{w}_{\delta,k}(s) & =\underline{w}(\delta2^{k}s)-\underline{w}(\delta2^{k}),
\end{align*}
and consider 
\[
\frac{\int_{E_{\lambda}\backslash E}e^{(n-2)w}|\nabla w|^{2}}{\int_{A_{\delta2^{k-1},\delta2^{k}}}e^{(n-2)\underline{w}}|\nabla\underline{w}|^{2}}=\frac{\int_{\delta^{-1}2^{-k}(E_{\lambda}\backslash E)}e^{(n-2)w_{\delta,k}}|\nabla_{\xi}w_{\delta,k}|^{2}}{\int_{A_{\frac{1}{2},1}}e^{(n-2)\underline{w}_{\delta,k}}|\nabla_{\xi}\underline{w}_{\delta,k}|^{2}}.
\]
By an inversion argument, as $k\to\infty$, $w_{\delta,k}(\xi),\underline{w}_{\delta,k}(|\xi|)\to\tilde{m}\log\frac{1}{|\xi|}$
in $W^{1,p}(A_{\frac{1}{2},1})$ for $1<p<n$, uniformly in $\delta\in(1/2,1]$.
Then for given small $\lambda>0$, there exists $k_{0}>0$ such that
when $k\ge k_{0}$, 
\[
\|\nabla_{\xi}w_{\delta,k}(\xi)-\nabla_{\xi}\underline{w}_{\delta,k}(|\xi|)\|_{L^{1}(A_{\frac{1}{2},1})}\le\lambda^{2}.
\]
Then 
\[
\text{meas}(\delta^{-1}2^{-k}(E_{\lambda}\backslash E))\le\frac{1}{\lambda}\int_{\delta^{-1}2^{-k}(E_{\lambda}\backslash E)}\left|\nabla_{\xi}w_{\delta,k}(\xi)-\nabla_{\xi}\underline{w}_{\delta,k}(|\xi|)\right|\le\lambda.
\]
Since on $\delta^{-1}2^{-k}(E_{\lambda}\backslash E)$, $w_{\delta,k}(\xi)=\underline{w}(\xi)+o(1)\le C$,
for $2<p<n$, 
\begin{align*}
\int_{\delta^{-1}2^{-k}(E_{\lambda}\backslash E)}e^{(n-2)w_{\delta,k}}|\nabla_{\xi}w_{\delta,k}|^{2} & \le C\text{meas}(\delta^{-1}2^{-k}(E_{\lambda}\backslash E))^{1-\frac{2}{p}}\|\nabla_{\xi}w_{\delta,k}\|^{2}_{L^{p}(A_{\frac{1}{2},1})}\\
 & \le C\lambda^{1-\frac{2}{p}}.
\end{align*}
Since
\[
\int_{A_{\frac{1}{2},1}}e^{(n-2)\underline{w}_{\delta,k}}|\nabla_{\xi}\underline{w}_{\delta,k}|^{2}>C>0
\]
for some uniform $C$, we know there exists $k_{0}$ such that when
$k>k_{0}$, for $\delta\in(\frac{1}{2},1]$, 
\[
\frac{\int_{\delta^{-1}2^{-k}(E_{\lambda}\backslash E)}e^{(n-2)w_{\delta,k}}|\nabla_{\xi}w_{\delta,k}|^{2}}{\int_{A_{\frac{1}{2},1}}e^{(n-2)\underline{w}_{\delta,k}}|\nabla_{\xi}\underline{w}_{\delta,k}|^{2}}\le C\lambda^{1-\frac{2}{p}},
\]
which implies 
\[
\frac{\int_{E_{\lambda}\backslash E}e^{(n-2)w}|\nabla w|^{2}}{\int_{A_{\frac{1}{2},1}}e^{(n-2)\underline{w}}|\nabla\underline{w}|^{2}}\le C\lambda^{1-\frac{2}{p}}.
\]

In the end we consider 
\begin{align*}
 & \left|\frac{\int_{A_{\delta2^{k-1},\delta2^{k}}\backslash(E\cup E_{\lambda})}e^{(n-2)w}|\nabla w|^{2}}{\int_{A_{\delta2^{k-1},\delta2^{k}}}e^{(n-2)\underline{w}}|\nabla\underline{w}|^{2}}-1\right|\\
= & \left|\frac{\int_{A_{\delta2^{k-1},\delta2^{k}}\backslash(E\cup E_{\lambda})}e^{(n-2)w}|\nabla w|^{2}-\int_{A_{\delta2^{k-1},\delta2^{k}}}e^{(n-2)\underline{w}}|\nabla\underline{w}|^{2}}{\int_{A_{\delta2^{k-1},\delta2^{k}}}e^{(n-2)\underline{w}}|\nabla\underline{w}|^{2}}\right|\\
= & \frac{\left|\int_{A_{\delta2^{k-1},\delta2^{k}}\backslash(E\cup E_{\lambda})}\left(e^{(n-2)w}|\nabla w|^{2}-e^{(n-2)\underline{w}}|\nabla\underline{w}|^{2}\right)\right|+\int_{E\cup E_{\lambda}}e^{(n-2)\underline{w}}|\nabla\underline{w}|^{2}}{\int_{A_{\delta2^{k-1},\delta2^{k}}}e^{(n-2)\underline{w}}|\nabla\underline{w}|^{2}}.
\end{align*}
Since $\text{meas}(E\cup E_{\lambda})/\text{meas}(A_{\delta2^{k-1},\delta2^{k}})$
is small, we may assume, for $k>k_{0}$, and $\delta\in(\frac{1}{2},1]$
\[
\frac{\int_{E\cup E_{\lambda}}e^{(n-2)\underline{w}}|\nabla\underline{w}|^{2}}{\int_{A_{\delta2^{k-1},\delta2^{k}}}e^{(n-2)\underline{w}}|\nabla\underline{w}|^{2}}<\lambda.
\]
 On the other hand 
\begin{align*}
\left|e^{(n-2)w}|\nabla w|^{2}-e^{(n-2)\underline{w}}|\nabla\underline{w}|^{2}\right| & \le|e^{(n-2)w}-e^{(n-2)\underline{w}}||\nabla w|^{2}+e^{(n-2)\underline{w}}||\nabla w|^{2}-|\nabla\underline{w}|^{2}|\\
 & \le o(1)e^{(n-2)\underline{w}}(|\nabla\underline{w}|+2^{-k}\lambda)^{2}+e^{(n-2)\underline{w}}2^{-k}\lambda(2|\nabla\underline{w}|+2^{-k}\lambda)\\
 & \le C2^{-2k}e^{(n-2)\underline{w}}(o(1)+\lambda),
\end{align*}
since 
\[
|\nabla\underline{w}|\le C2^{-k},\forall x\in A_{\delta2^{k-1},\delta2^{k}}.
\]
Then for $k>k_{0}$, $\delta\in(\frac{1}{2},1]$, 
\[
\frac{\left|\int_{A_{\delta2^{k-1},\delta2^{k}}\backslash(E\cup E_{\lambda})}\left(e^{(n-2)w}|\nabla w|^{2}-e^{(n-2)\underline{w}}|\nabla\underline{w}|^{2}\right)\right|}{\int_{A_{\delta2^{k-1},\delta2^{k}}}e^{(n-2)\underline{w}}|\nabla\underline{w}|^{2}}<C(\lambda+o(1)).
\]
Then we proved the Lemma \ref{gradient ratio go to 1}.

\end{proof}

\begin{lem}\label{f'(yN) small}

Let $f$ be a positive continuous function on $[R_{0},\infty)$. Suppose
that its finite one-sided derivatives exist everywhere, that
\[
f'(r-0)\le f'(r+0),
\]
with strict inequality at most at countably many points, and that
\[
\lim_{r\to\infty}f(r)=1.
\]
Then there is an increasing sequence $r_{j}\to\infty$ such that $f$
is differentiable at every $r_{j}$ and
\[
r_{j}f'(r_{j})\to0,\qquad\frac{r_{j+1}}{r_{j}}\to1.
\]

\end{lem}

\begin{proof}Define
\[
F(s)=f(e^{s}).
\]
The assumptions on the one-sided derivatives are preserved under this
change of variables, and
\[
F(s)\to1\qquad\text{as }s\to\infty.
\]

We first record the local selection argument that will be used below.
Let $I=[a,b]$ be a compact interval. We claim that there is a point
$c\in(a,b)$ at which $F$ is differentiable and
\[
|F'(c)|\le\frac{2\osc_{I}F}{b-a}.
\]

Indeed, if $F$ has a local maximum at an interior point $c$, then
\[
F'(c-0)\ge0,\qquad F'(c+0)\le0.
\]
The assumed inequality $F'(c-0)\le F'(c+0)$ forces both one-sided
derivatives to be zero. Thus $F$ is differentiable at $c$ and the
claim holds. If $F$ has no local maximum in $(a,b)$, then it is
monotone on a subinterval $J\subset I$ of length at least $(b-a)/2$;
the only possible change of monotonicity is through a local minimum.
The one-sided mean-value inequality on $J$ gives a differentiability
point $c\in J$ such that
\[
|F'(c)|\le\frac{\osc_{J}F}{|J|}\le\frac{2\osc_{I}F}{b-a}.
\]
This proves the local claim.

Choose an increasing sequence $S_{k}\to\infty$ such that
\[
S_{k+1}-S_{k}\ge1,\qquad\sup_{s\ge S_{k}}|F(s)-1|\le k^{-3}.
\]
Partition every interval $[S_{k},S_{k+1}]$ into finitely many consecutive
subintervals $I_{k,m}$ whose lengths satisfy
\[
\frac{1}{2k}\le|I_{k,m}|\le\frac{1}{k}.
\]
For example, one may divide $[S_{k},S_{k+1}]$ into $\lceil k(S_{k+1}-S_{k})\rceil$
equal pieces.

On every $I_{k,m}$,
\[
\osc_{I_{k,m}}F\le2k^{-3}.
\]
Applying the local claim gives a differentiability point $s_{k,m}\in I_{k,m}$
such that
\[
|F'(s_{k,m})|\le\frac{4k^{-3}}{|I_{k,m}|}\le8k^{-2}.
\]
List all these points in increasing order as $s_{j}$. The size of
the partition intervals implies
\[
s_{j+1}-s_{j}\to0,\qquad F'(s_{j})\to0.
\]
Set $r_{j}=e^{s_{j}}$. Then
\[
\frac{r_{j+1}}{r_{j}}=e^{s_{j+1}-s_{j}}\to1,\qquad r_{j}f'(r_{j})=F'(s_{j})\to0.
\]

\end{proof}

\begin{lem}\label{v v' ratio go to 1}

Assume that $g=e^{2w}|dx|^{2}$ is nonflat. Let
\[
h(r)=\fint_{\mathbb{S}^{n-1}}v^{2}(r\Theta)dS(\Theta),\qquad\underline{h}(r)=\underline{v}^{2}(r)=e^{(n-2)\underline{w}(r)}.
\]
Then there exists an increasing sequence $r_{j}\to\infty$ such that
\[
\frac{r_{j+1}}{r_{j}}\to1
\]
and
\[
\frac{\int_{\partial B_{r_{j}}(0)}v\frac{\partial v}{\partial\nu}dS}{\int_{\partial B_{r_{j}}(0)}\underline{v}\frac{\partial\underline{v}}{\partial\nu}dS}\to1.
\]
Moreover, if
\[
l_{j}=L(r_{j}),\qquad L(r)=\int^{r}_{0}e^{\underline{w}(t)}dt,
\]
then
\[
\frac{l_{j+1}}{l_{j}}\to1.
\]

\end{lem}

\begin{proof}By Lemma \ref{area ratio go to 1},
\[
f(r):=\frac{h(r)}{\underline{h}(r)}=\frac{\fint_{\mathbb{S}^{n-1}}e^{(n-2)w(r\Theta)}dS(\Theta)}{e^{(n-2)\underline{w}(r)}}\to1.
\]

The function $h$ is continuously differentiable. The radial function
$\underline{h}$ is positive and continuous, its finite one-sided
derivatives exist, and
\[
\underline{h}'(r-0)\ge\underline{h}'(r+0),
\]
with strict inequality at most at countably many points. Hence
\[
f'(r-0)-f'(r+0)=-\frac{h(r)}{\underline{h}(r)^{2}}\left(\underline{h}'(r-0)-\underline{h}'(r+0)\right)\le0.
\]
Thus $f$ satisfies the assumptions of Lemma \ref{f'(yN) small}.
We obtain an increasing sequence $r_{j}$ such that
\[
r_{j}f'(r_{j})\to0,\qquad\frac{r_{j+1}}{r_{j}}\to1.
\]
At these points $\underline{h}$ is also differentiable, since a strict
jump of its one-sided derivatives would produce a strict jump of the
one-sided derivatives of $f$.

Since the metric is nonflat, Theorem 3.1 gives $\widetilde{m}>0$.
The radial asymptotics give
\[
r\underline{w}'(r)\to-\widetilde{m}
\]
at differentiability points. Therefore
\[
r(\log\underline{h})'(r)=(n-2)r\underline{w}'(r)\to-(n-2)\widetilde{m}\ne0.
\]

At every point where the relevant derivatives exist, the correct identity
is
\[
f'(r)=(\log\underline{h})'(r)\left(\frac{h'(r)}{\underline{h}'(r)}-f(r)\right).
\]
Equivalently,
\[
\frac{h'(r)}{\underline{h}'(r)}=f(r)+\frac{rf'(r)}{r(\log\underline{h})'(r)}.
\]
Consequently,
\[
\frac{h'(r_{j})}{\underline{h}'(r_{j})}\to1.
\]

Finally,
\[
\int_{\partial B_{r}(0)}v\frac{\partial v}{\partial\nu}dS=\frac{r^{n-1}|\mathbb{S}^{n-1}|}{2}h'(r),
\]
and the corresponding identity holds for $\underline{v}$ and $\underline{h}$.
The common factor cancels, proving the boundary-flux ratio.

It remains to compare the radial geodesic radii. The radial asymptotics
used in the proof of Theorem \ref{main thm2 volume ratio} give
\[
\frac{re^{\underline{w}(r)}}{L(r)}\to1-\widetilde{m}.
\]
Moreover, $\underline{w}$ is eventually nonincreasing. Thus
\begin{align*}
0 & \le\frac{L(r_{j+1})-L(r_{j})}{L(r_{j})}\\
 & =\frac{\int^{r_{j+1}}_{r_{j}}e^{\underline{w}(t)}dt}{L(r_{j})}\\
 & \le\left(\frac{r_{j+1}}{r_{j}}-1\right)\frac{r_{j}e^{\underline{w}(r_{j})}}{L(r_{j})}\to0.
\end{align*}
Therefore $l_{j+1}/l_{j}\to1$.

\end{proof}

\begin{lem}\label{area ratio go to 1}
\[
\lim_{r\to\infty}\frac{\fint_{\mathbb{S}^{n-1}(0)}e^{(n-2)w(r\Theta)}dS(\Theta)}{e^{(n-2)\underline{w}(r)}}=1.
\]

\end{lem}

\begin{proof}
\begin{align*}
 & \frac{\fint_{\mathbb{S}^{n-1}(0)}e^{(n-2)w(r\Theta)}dS(\Theta)}{e^{(n-2)\underline{w}(r)}}\\
= & \fint_{\mathbb{S}^{n-1}(0)}e^{(n-2)(w(r\Theta)-\underline{w}(r))}dS(\Theta)\\
= & \frac{1}{|\mathbb{S}^{n-1}(0)|}\left(\int_{\mathbb{S}^{n-1}(0)\backslash r^{-1}(E)}+\int_{\mathbb{S}^{n-1}(0)\cap r^{-1}(E)}\right)e^{(n-2)(w(r\Theta)-\underline{w}(r))}dS(\Theta)\ge1.
\end{align*}
For $\varepsilon>0$, there exists $M$ such that when $r>M$,
\[
\int_{\mathbb{S}^{n-1}(0)\backslash r^{-1}(E)}e^{(n-2)(w(r\Theta)-\underline{w}(r))}dS(\Theta)\le\int_{\mathbb{S}^{n-1}(0)\backslash r^{-1}(E)}e^{o(1)}dS(\Theta)\le(1+\varepsilon)|\mathbb{S}^{n-1}(0)|.
\]
For the strong $\mathcal{E}$-set $E$, we may assume 
\[
\mathbb{S}^{n-1}(0)\cap r^{-1}(E)=\bigcup_{i}B^{g_{\mathbb{S}^{n-1}}}_{r_{i}}(p_{i}),p_{i}\in\mathbb{S}^{n-1}.
\]
for $\sum_{i}|B^{g_{\mathbb{S}^{n-1}}}_{r_{i}}(p_{i})|$ small. 

Notice that $w_{r}(\xi)=w(r\xi)-\underline{w}(5r)\ge0$ in $\xi\in A_{\frac{1}{4},4}$.
Since as $r\to\infty$, $w_{r}(\xi)\to\tilde{m}\log\frac{5}{|\xi|}$
in $W^{1,p}(A_{\frac{1}{4},4})$, it is easy to show that $w_{r}(\xi)\le C_{4}(1+W^{\mu_{r}}_{1,n}(\xi,\frac{1}{8}))$,
where $\mu=-\Delta_{n}w$ and $\mu_{r}(D)=\mu(rD)$. Then for $\Theta\in\mathbb{S}^{n-1}(0)$,
$w_{r}(\Theta)\le C_{4}(1+W^{\mu_{r}}_{1,n}(\Theta,\frac{1}{4}))$.
Notice, for $\lambda>2C_{4}$, 
\[
|\{\Theta\in B^{g_{\mathbb{S}^{n-1}}}_{r_{i}}(p_{i});w_{r}(\Theta)>\lambda\}|\le\left|\left\{ \Theta\in B^{g_{\mathbb{S}^{n-1}}}_{r_{i}}(p_{i});W^{\mu_{r}}_{1,n}(\Theta,\frac{1}{8})\ge\frac{\lambda}{2C_{4}}\right\} \right|.
\]

We let 
\[
\begin{cases}
-\Delta^{\xi}_{n}v(\xi)=\mu_{r}\lfloor A_{\frac{3}{8},3}, & \xi\in A_{\frac{1}{4},4},\\
v(\xi)=0, & \xi\in\partial A_{\frac{1}{4},4}.
\end{cases}
\]
From Lemma \ref{capacity estimate in terms of measure}, we have 
\[
\text{cap}_{n}\left(\{\xi\in A_{\frac{1}{4},4};v(\xi)\ge s\},\mathring{A}_{\frac{1}{4},4}\right)\le\frac{C\mu_{r}(A_{\frac{3}{8},3})}{s^{n-1}}.
\]
On the other hand, from \cite[Theorem 1.6]{KM1994}, for $\xi\in A_{\frac{1}{2},2},$
\[
CW^{\mu_{r}}_{1,n}(\xi,\frac{1}{8})\le v(\xi).
\]
Then 
\[
\text{cap}_{n}\left(\left\{ \xi\in A_{\frac{1}{2},2};W^{\mu_{r}}_{1,n}(\xi,\frac{1}{8})\ge\frac{\lambda}{2C_{4}}\right\} ,\mathring{A}_{\frac{1}{4},4}\right)\le\frac{C\mu_{r}(A_{\frac{3}{8},3})}{\lambda^{n-1}}.
\]
And hence 
\begin{equation}
\text{cap}_{n}\left(\left\{ \Theta\in B^{g_{\mathbb{S}^{n-1}}}_{r_{i}}(p_{i});W^{\mu_{r}}_{1,n}(\Theta,\frac{1}{8})\ge\frac{\lambda}{2C_{4}}\right\} ,\mathring{A}_{\frac{1}{4},4}\right)\le\frac{C\mu_{r}(A_{\frac{3}{8},3})}{\lambda^{n-1}}.\label{capacity of B_ri  upper bound}
\end{equation}

On the other hand, we calculate the lower bound of the capacity $\text{cap}_{n}\left(D\subset\mathbb{S}^{n-1},\mathring{A}_{\frac{1}{4},4}\right)$
in term of its $n-1$ dimensional Hausdorff measure. We let 
\[
\begin{cases}
-\Delta^{\xi}_{n}\tilde{v}(\xi)=\mathcal{H}^{n-1}\lfloor D & \xi\in A_{\frac{1}{4},4},\\
\tilde{v}(\xi)=0 & \xi\in\partial A_{\frac{1}{4},4}.
\end{cases}
\]
Then we know 
\[
\text{cap}_{n}(D,\mathring{A}_{\frac{1}{4},4})\ge\frac{|D|}{(\sup\tilde{v})^{n-1}}.
\]
Obviously $\sup\tilde{v}=\sup_{D}\tilde{v}$. For $\xi\in D$, 
\begin{align*}
\tilde{v}(\xi) & \le C\inf_{B_{\frac{1}{4}}(\xi)}\tilde{v}+CW^{\mathcal{H}^{n-1}\lfloor D}_{1,n}(\xi,\frac{1}{2})\\
 & \le C\left(\frac{|D|}{\text{cap}_{n}(B_{\frac{1}{4}}(\xi),\mathring{A}_{\frac{1}{4},4})}\right)^{\frac{1}{n-1}}+CW^{\mathcal{H}^{n-1}\lfloor D}_{1,n}(\xi,\frac{1}{2})\\
 & \le C|D|^{\frac{1}{n-1}}+CW^{\mathcal{H}^{n-1}\lfloor D}_{1,n}(\xi,\frac{1}{2}).
\end{align*}
Note that 
\begin{align*}
W^{\mathcal{H}^{n-1}\lfloor D}_{1,n}(\xi,\frac{1}{2}) & =\int^{\frac{1}{2}}_{0}\mathcal{H}^{n-1}(D\cap B_{t}(\xi))^{\frac{1}{n-1}}\frac{dt}{t}\\
 & \le\int^{\left(\frac{|D|}{\omega_{n-1}}\right)^{\frac{1}{n-1}}}_{0}\omega^{\frac{1}{n-1}}_{n-1}dt+\int^{\frac{1}{2}}_{\left(\frac{|D|}{\omega_{n-1}}\right)^{\frac{1}{n-1}}}|D|^{\frac{1}{n-1}}\frac{dt}{t}\\
 & \le C|D|^{\frac{1}{n-1}}(1+\log\frac{1}{|D|}).
\end{align*}
So 
\[
\sup\tilde{v}\le C|D|^{\frac{1}{n-1}}(1+\log\frac{1}{|D|})
\]
and hence 
\[
\text{cap}_{n}(D,\mathring{A}_{\frac{1}{4},4})\ge\frac{C}{(1+\log\frac{1}{|D|})^{n-1}}.
\]
Then together with (\ref{capacity of B_ri  upper bound}), we have
\[
\left|\left\{ \Theta\in B^{g_{\mathbb{S}^{n-1}}}_{r_{i}}(p_{i});W^{\mu_{r}}_{1,n}(\Theta,\frac{1}{8})\ge\frac{\lambda}{2C_{4}}\right\} \right|\le\exp\left(1-\frac{\lambda}{C\mu_{r}(A_{\frac{3}{8},3})^{\frac{1}{n-1}}}\right).
\]
 Then for $\lambda>2C_{4}$,
\[
|\{\Theta\in B^{g_{\mathbb{S}^{n-1}}}_{r_{i}}(p_{i});w_{r}(\Theta)>\lambda\}|\le\exp\left(1-\frac{\lambda}{C\mu_{r}(A_{\frac{3}{8},3})^{\frac{1}{n-1}}}\right).
\]
 Now 
\begin{align*}
 & \int_{B^{g_{\mathbb{S}^{n-1}}}_{r_{i}}(x_{i})}e^{(n-2)(w(r\Theta)-\underline{w}(r))}dS(\Theta)\\
\le & C\int_{B^{g_{\mathbb{S}^{n-1}}}_{r_{i}}(x_{i})}e^{(n-2)w_{r}(\Theta)}dS(\Theta)\\
\le & C\int^{+\infty}_{0}\left|\left\{ w_{r}(\Theta)>\frac{1}{n-2}\log t\right\} \right|dt\\
\le & Cr^{n-1-C\mu_{r}(A_{\frac{3}{8},3})^{\frac{1}{n-1}}}_{i}+\int^{\infty}_{r^{-(n-1)C\mu_{r}(A_{\frac{3}{8},3})^{\frac{1}{n-1}}}_{i}}\exp\left(1-\frac{\log t}{C\mu_{r}(A_{\frac{3}{8},3})^{\frac{1}{n-1}}}\right)dt\\
\le & Cr^{n-1-C\mu_{r}(A_{\frac{3}{8},3})^{\frac{1}{n-1}}}_{i}.
\end{align*}
Then from (\ref{Strong epsilon set}), as $r\to\infty$, 
\[
\sum\int_{B^{g_{\mathbb{S}^{n-1}}}_{r_{i}}(x_{i})}e^{(n-2)(w(r\Theta)-\underline{w}(r))}dS(\Theta)\to0.
\]
Then we proved the lemma.

\end{proof}

Then we let $r_{i}\to+\infty$ be the sequence in Lemma \ref{v v' ratio go to 1}.
Since $\int_{\partial B_{r_{i}}(0)}\frac{\partial\underline{v}}{\partial\nu}\underline{v}dS(\partial B^{e}_{t_{i}})\le0$,
we have
\[
\frac{\int_{B_{r_{i}}(0)}|\nabla v|^{2}-\int_{\partial B_{r_{i}}(0)}\frac{\partial v}{\partial\nu}vdS(\partial B^{e}_{t_{i}})}{\int_{B_{r_{i}}(0)}|\nabla\underline{v}|^{2}-\int_{\partial B_{r_{i}}(0)}\frac{\partial\underline{v}}{\partial\nu}\underline{v}dS(\partial B^{e}_{t_{i}})}\to1
\]
 as $r_{i}\to\infty.$ For 
\[
l_{i}=\int^{r_{i}}_{0}e^{\underline{w}(t)}dt,
\]
 we know 
\begin{equation}
\frac{\int_{B^{\underline{g}}_{l_{i}}(0)}Rd\mu_{g}}{\int_{B^{\underline{g}}_{l_{i}}(0)}\underline{R}d\mu_{\underline{g}}}\to1,{\rm as}\,l_{i}\to\infty.\label{R integral =00005CunderlineRintegralratioin the same domain}
\end{equation}
 Then 
\[
\lim_{l_{i}\to\infty}\frac{\int_{B^{\underline{g}}_{l_{i}}(0)}Rd\mu_{g}}{l^{n-2}_{i}}=\lim_{l_{i}\to\infty}\frac{\int_{B^{\underline{g}}_{l_{i}}(0)}\underline{R}d\mu_{\underline{g}}}{l^{n-2}_{i}}.
\]
 From Lemma \ref{v v' ratio go to 1}, we know $\lim_{l_{i}\to\infty}l_{i+1}/l_{i}=1$,
which implies 
\begin{equation}
\lim_{l\to\infty}\frac{\int_{B^{\underline{g}}_{l}(0)}Rd\mu_{g}}{l^{n-2}}=\lim_{l\to\infty}\frac{\int_{B^{\underline{g}}_{l}(0)}\underline{R}d\mu_{\underline{g}}}{l^{n-2}}=\begin{cases}
(n-1)|\mathbb{S}^{n-1}|(\nu^{\frac{n-3}{n-1}}-\nu) & n>3,\\
8\pi(1-\nu) & n=3.
\end{cases}\label{R limit go to the right one}
\end{equation}
In the end we  prove that 
\[
\lim_{l\to\infty}\frac{\int_{B^{g}_{l}(0)}Rd\mu_{g}}{l^{n-2}}=\lim_{l\to\infty}\frac{\int_{B^{\underline{g}}_{l}(0)}Rd\mu_{g}}{l^{n-2}}.
\]
 It is obvious that 
\[
\frac{\int_{B^{g}_{l}(0)}Rd\mu_{g}}{\int_{B^{\underline{g}}_{l}(0)}Rd\mu_{g}}\le1.
\]
 For opposite, from (\ref{B^g contains B^=00005Cunderlineg_g}), we
have 
\[
\lim_{l\to\infty}\frac{\int_{B^{g}_{l}(0)}Rd\mu_{g}}{l^{n-2}}\ge\lim_{l\to\infty}\frac{\int_{B^{\underline{g}}_{l}(0)}Rd\mu_{g}}{l^{n-2}}.
\]
Then Theorem \ref{main thm 3} follows.

\begin{proof}(of Theorem \ref{main thm 4 generalization of Chen-Zhu's theorem})
We only need to consider the case when $(M,g)$ is isometric to $(\mathbb{R}^{n},e^{2w}|dx|^{2})$,
since in Case 3 of Theorem \ref{Carron-Herzlich classification thm},
the scalar curvature is a positive constant, which does not satisfy
(\ref{Scalar curvature decay assumption 1}) or (\ref{Scalar curvature decay assumption 2}).
From Theorem \ref{main thm1 refined analytic result of u},
\[
-\liminf_{x\to\infty}\frac{w(x)}{\log|x|}=\tilde{m}=1-\nu{}^{\frac{1}{n-1}}.
\]
 Then from Theorem \ref{main thm 3}, 
\[
\lim_{l\to\infty}\frac{\int_{B^{g}_{l}(0)}R_{g}}{l^{n-2}}=\begin{cases}
(n-1)(\nu^{\frac{n-3}{n-1}}-\nu)|\mathbb{S}^{n-1}|, & n>3\\
8\pi(1-\nu), & n=3.
\end{cases}
\]
From Theorem \ref{main thm2 volume ratio}, 
\[
\lim_{l\to\infty}\frac{V_{g}(B^{g}_{l}(0))}{\omega_{n}l^{n}}=\nu.
\]

If $\nu>0$, then 
\[
\lim_{l\to\infty}\frac{l^{2}\int_{B^{g}_{l}(0)}R_{g}}{V_{g}(B^{g}_{l}(0))}=\begin{cases}
\frac{(n-1)}{\omega_{n}}(\nu^{\frac{-2}{n-1}}-1)|\mathbb{S}^{n-1}|, & n>3\\
6(\frac{1}{\nu}-1), & n=3.
\end{cases}
\]
If there holds 
\begin{equation}
\frac{\int_{B^{g}_{l}(0)}R_{g}}{V_{g}(B^{g}_{l}(0))}=o(\frac{1}{l^{2}})\label{Rg decay condition 1}
\end{equation}
or 
\begin{equation}
\int^{\infty}_{0}\frac{l}{V_{g}(B^{g}_{l}(0))}\left(\int_{B^{g}_{l}(0)}R_{g}\right)dl<+\infty,\label{Rg decay condition 2}
\end{equation}
then $\nu$ has to be $1$, which implies that $(M,g)$ is isometric
to $(\mathbb{R}^{n},|dx|^{2})$. 

Now we consider $\nu=0$ case. 

\begin{lem}

If $\nu=0$ and either (\ref{Rg decay condition 1}) or (\ref{Rg decay condition 2})
holds, then either
\[
\frac{\int_{B^{\underline{g}}_{l}(0)}R_{\underline{g}}}{V_{\underline{g}}(B^{\underline{g}}_{l}(0))}=o(\frac{1}{l^{2}})
\]
or 
\[
\int^{\infty}_{0}\frac{l}{V_{\underline{g}}(B^{\underline{g}}_{l}(0))}\left(\int_{B^{\underline{g}}_{l}(0)}R_{\underline{g}}\right)dl<+\infty,
\]
holds respectively.

\end{lem}

\begin{proof}

First we will show that for every fixed $a\in(0,1)$, there exists
$C_{a}>0$ such that 
\begin{equation}
V_{\underline{g}}(B^{\underline{g}}_{l}(0))\le C(a)V_{\underline{g}}(B^{\underline{g}}_{al}(0)).\label{slow variation of V_=00005Cunderlineg}
\end{equation}
Indeed since $\psi(s)/s$ is nonincreasing, 
\[
V_{\underline{g}}(B^{\underline{g}}_{al}(0))\ge Cal\psi(al)^{n-1},
\]
while 
\[
V_{\underline{g}}(B^{\underline{g}}_{l}(0))-V_{\underline{g}}(B^{\underline{g}}_{al}(0))\le(1-a)l\psi(l)^{n-1}.
\]

On one hand, from (\ref{slow variation of V_=00005Cunderlineg})
\begin{equation}
V_{g}(B^{g}_{l}(0))\subset V_{g}(B^{\underline{g}}_{l}(0))\le CV_{\underline{g}}(B^{\underline{g}}_{l}(0))\le C(a)V_{\underline{g}}(B^{\underline{g}}_{al}(0)).\label{volume transfered to =00005Cunderlineg}
\end{equation}

On the other hand we assume $l_{i}\le al<l_{i+1}$. So for $l$ large,
$l_{i}/l_{i+1}$ is close to $1$ and so
\[
\frac{l}{l_{i+1}}=\frac{1}{a}\frac{al}{l_{i}}\frac{l_{i}}{l_{i+1}}\ge\frac{2}{1+a}>1.
\]
So $B^{\underline{g}}_{l_{i+1}}(0)\subset B^{g}_{l}(0)$. Then 
\begin{equation}
\int_{B^{g}_{l}}R_{g}\ge\int_{B^{\underline{g}}_{l_{i+1}}}R_{g}\ge C\int_{B^{\underline{g}}_{l_{i+1}}}R_{\underline{g}}\ge C\int_{B^{\underline{g}}_{al}}R_{\underline{g}}.\label{scalar curvature integral transfered to =00005Cunderlineg}
\end{equation}

So from (\ref{volume transfered to =00005Cunderlineg})(\ref{scalar curvature integral transfered to =00005Cunderlineg}),
we see that, there holds
\[
\frac{\int_{B^{\underline{g}}_{al}}R_{\underline{g}}}{V_{\underline{g}}(B^{\underline{g}}_{al}(0))}=o\left(\frac{1}{l^{2}}\right)
\]
or 
\[
\int^{\infty}_{0}\frac{al}{V_{\underline{g}}(B^{\underline{g}}_{al}(0))}\left(\int_{B^{\underline{g}}_{al}(0)}R_{\underline{g}}\right)dl<+\infty,
\]
respectively, which ends the proof. 

\end{proof}

Then from (\ref{=00005CunderlineR's expression}), we have, as $r\to\infty$,
\[
\frac{l^{2}\int^{l}_{0}\left((n-2)\frac{1-\psi'(s)^{2}}{\psi(s)^{2}}-2\frac{\psi''(s)}{\psi(s)}\right)\psi(s)^{n-1}ds}{\int^{l}_{0}\psi(s)^{n-1}ds}=o(1)
\]
or 
\[
\int^{\infty}_{0}\frac{l\int^{l}_{0}\left((n-2)\frac{1-\psi'(s)^{2}}{\psi(s)^{2}}-2\frac{\psi''(s)}{\psi(s)}\right)\psi(s)^{n-1}ds}{\int^{l}_{0}\psi(s)^{n-1}ds}dl<+\infty.
\]
From Lemma \ref{psi theta information}, $(n-2)\frac{1-\psi'(s)^{2}}{\psi(s)^{2}}$
and $-2\frac{\psi''(s)}{\psi(s)}$ are both nonnegative. We must have,
as $l\to\infty$
\begin{equation}
\frac{l^{2}\int^{l}_{0}(1-\psi'(s)^{2})\psi(s)^{n-3}ds}{\int^{l}_{0}\psi(s)^{n-1}ds}=o(1)\label{equivalent condition 1}
\end{equation}
or 
\begin{equation}
\int^{\infty}_{0}\frac{l\int^{l}_{0}\frac{1-\psi'(s)^{2}}{\psi(s)^{2}}\psi(s)^{n-1}ds}{\int^{l}_{0}\psi(s)^{n-1}ds}dl<+\infty.\label{equivalent condtion 1}
\end{equation}

However, 
\begin{align*}
 & \lim_{l\to\infty}\frac{l^{2}\int^{l}_{0}(1-\psi'(s)^{2})\psi(s)^{n-3}ds}{\int^{l}_{0}\psi(s)^{n-1}ds}\\
= & \lim_{l\to\infty}\frac{l^{2}(1-\psi'(l)^{2})\psi(l)^{n-3}+2l\int^{l}_{0}(1-\psi'(s)^{2})\psi(s)^{n-3}ds}{\psi(l)^{n-1}}\\
\ge & \lim_{l\to\infty}\frac{l^{2}(1-\psi'(l)^{2})}{\psi(l)^{2}}\\
= & +\infty,
\end{align*}
 as when $\nu=0$, we have 
\begin{align*}
\lim_{l\to\infty}\frac{l}{\psi(l)} & =+\infty,\\
\lim_{l\to\infty}\psi'(l) & =0.
\end{align*}
So both (\ref{equivalent condition 1}) and (\ref{equivalent condtion 1})
fail and we finish the proof.

\end{proof}

\paragraph{Acknowledgements:}

The idea of this paper originates from the author\textquoteright s
prior collaborative research with Professor Jie Qing. The author is
deeply grateful to Professor Jie Qing for numerous profound insights.
Special thanks are extended to Professors Gang Tian, Xiping Zhu and
Xiao Zhong for their valuable comments, as well as to Yalong Shi,
Mingxiang Li, Mijia Lai, Jialong Deng, Guoyi Xu and Chi Li for their
helpful feedback. The author is supported by the Natural Science Foundation
of Tianjin, No. 22JCJQJC00130 and Fundamental Research Funds for the
Central Universities.


\begin{thebibliography}{CQY2000A}
\bibitem[AH1996]{AH1996}Adams, D.R. and Hedberg, L. I. (1996): Function
spaces and potential theory (Vol. 314). Springer. https://doi.org/10.1007/978-3-662-03282-4

\bibitem[CH2006]{CH2006}Carron, G. and Herzlich, M.: Conformally
flat manifolds with nonnegative Ricci curvature. Compos. Math. 142(3),
798\textendash 810 (2006)

\bibitem[CHY2004]{CHY2004}Chang, S.-Y.A., Hang, F.B. and Yang, P.:
On a class of locally conformally flat manifolds. IMRN 4, 185\textendash 209
(2004)

\bibitem[CQY2000A]{CQY2000A}Chang, S.-Y.A., Qing, J. and Yang, P.:
On the Chern-Gauss-Bonnet integral for conformal metrics on $\mathbb{R}^{4}$.
Duke Math. J. 103, No. 3, 523-544 (2000)

\bibitem[CQY2000B]{CQY2000B}Chang, S.-Y.A., Qing, J. and Yang, P.:
Compactification of a class of conformally flat 4-manifold. Invent.
Math. 142:1 (2000), 65\textendash 93

\bibitem[CV1935]{CV1935} Cohn-Vossen S.: Kürzeste Wege und Totalkrümmung
auf Flächen. Compositio Math. 2(1935),69\textendash 133

\bibitem[CZ2002]{CZ2002}Chen, B. and Zhu, X.: A gap theorem for complete
noncompact manifolds with nonnegative Ricci curvature. Commun. Geom.
Anal. 10(1), 217\textendash 239 (2002)

\bibitem[DMOP1999]{DMOP1999}Dal Maso G., Murat F., Orsina L. and
Prignet A.: Renormalized solutions of elliptic equations with general
measure data. Ann. Scuola Norm. Sup. Pisa Cl. Sci. 28, pp. 741\textendash 808(1999)

\bibitem[De2026]{De2026}Deng J.:Cohn-{}-Vossen-Type Inequalities
for Three-Manifolds and Locally Conformally Flat Manifolds. arXiv:2606.01368

\bibitem[DM2011]{DM2011}Duzaar, F. and Mingione, G.: Gradient estimates
via non-linear potentials. Amer. J. Math. 133 (2011), no. 4, 1093\textendash 1149

\bibitem[Fi1965]{Finn1965}Finn, R.: On a class of conformal metrics,
with application to differential geometry in the large. Comment. Math.
Helv.40 (1965),1\textendash 30

\bibitem[Gh2025]{Gh2025} Ghatasheh, A.: Mean value theorems and L'Hospital-type
rules for regulated functions. arXiv:2305.17572v2

\bibitem[GPZ1994]{GPZ1994}Greene, R., Petersen, P., and Zhu S.: Riemannian
manifolds of faster-than-quadratic curvature decay. Int. Math. Res.
Not. 9 (1994)

\bibitem[GW1982]{GW1982} Greene, R.E. and Wu, H., Gap theorems for
noncompact Riemannian manifolds. Duke Math. J. 49 (1982), 731\textendash 756

\bibitem[Ha1960]{Ha1960} Hayman, W. K.: Slowly growing integral and
subharmonic functions. Comment Math. Helv. 34 (1960) 75-84

\bibitem[HK1976]{HK1976}Hayman, W.K. and Kennedy, P.B.: Subharmonic
Functions, vol. 1. Academic Press, London, New York, San Francisco
(1976)

\bibitem[HKM1993]{HKM1993}Heinonen, J., Kilpeläinen, T. and Martio,
O.: Nonlinear Potential Theory of Degenerate Elliptic Equations. Oxford
Univ. Press, Oxford (1993)

\bibitem[HL2025]{HL2025}Hsiao, M., Lee, M.: Gap theorem on locally
conformally flat manifold. arXiv:2504.08189v1

\bibitem[Hu1957]{Huber1957}Huber, A.: On subharmonic functions and
differential geometry in the large. Comment.Math.Helv.32 (1957),13\textendash 72

\bibitem[KM1994]{KM1994}Kilpeläinen, T. and Malý, J.: The Wiener
test and potential estimates for quasilinear elliptic equations. Acta.
Math. 172, 137\textendash 161 (1994)

\bibitem[KV1986]{KV1986} Kichenassamy, S. and Veron, L.: Singular
solutions of the $p$-Laplace equation. Math. Ann. 275, 599\textendash 615
(1986)

\bibitem[LWX2025]{LWX2025} Li, M., Wei,J. and Xu, X.: On geometry
of $Q^{(2k)}_{g}$-curvature. arXiv:2506.20165v1

\bibitem[Liu2022]{Liu2022}Liu, G. Cohn\textendash Vossen inequality
on certain noncompact Kähler manifolds. Mathematische Zeitschrift
(2022) 302:1025\textendash 1034 

\bibitem[LMQZ2025]{LMQZ2025} Liu, H., Ma, S., Qing, J. and Zhong,
S.: On the asymptotic behavior of $p$-superharmonic functions at
singularities. To appear in Calc.Var.Par.Diff.Equ.

\bibitem[Ma2016]{Ma2016}Ma, L., Gap theorems for locally conformally
flat manifolds. J. Differential Equations 260 (2016), no. 2, 1414\textendash 1429

\bibitem[Maz2011]{Maz2011}Mazya, V. G. and Shaposhnikova T. O.: Sobolev
Spaces (2nd ed.){[}M{]}. Heidelberg: Springer, 2011

\bibitem[MQ2021]{MQ2021}Ma, S. and Qing, J.: On $n$-superharmonic
functions and some geometric applications. Calc. Var. Partial Differential
Equations 60:6 (2021)

\bibitem[MQ2022]{MQ2022}Ma, S. and Qing, J.: On Huber-type theorems
in general dimensions. Adv. Math. 395 (2022)

\bibitem[MSY1981]{MSY1981}Mok, N., Siu, Y. and Yau, S.: The Poincaré-Lelong
equation on complete Kähler manifolds. Compositio Math. 44 (1981),
no. 1-3, 183\textendash 218

\bibitem[MW2025]{MW2025} Ma, S. and Wang, Z.: $N$-Laplacian and
$N/2$-Hessian Type Equations with Exponential Reaction Terms and
Measure Data. Potential Analysis (2026) 64:27

\bibitem[MW2025]{MuWa2025}Munteanu, O. and Wang, J.: Sharp integral
bound of scalar curvature on $3$-manifolds. Arxiv:2505.10520v1

\bibitem[Ni2012]{Ni2012}Ni, L.: An optimal gap theorem. Invent. Math.
189 (2012), no. 3, 737\textendash 761

\bibitem[Se1964]{Se1964} Serrin, J.: Local behavior of solutions
of quasi-linear equations. Acta Math. 111 (1964) 247\textendash 302

\bibitem[Se1965]{Se1965} Serrin, J.: Isolated singularities of solutions
of quasi-linear equations. Acta Math. 113 (1965) 219\textendash 240

\bibitem[Ve2017]{Veron2017}Veron, L.: Local and Global Aspects of
Quasilinear Degenerate Elliptic Equations. World Scientific, New Jersey
(2017)

\bibitem[Xu2020]{Xu2020}Xu, G.: Integral of scalar curvature on non-parabolic
manifolds. J. Geom. Anal. 30 (2020), no. 1, 901\textendash 909, https://doi.org/10.1007/s12220-019-00174-7,
DOI 10.1007/s12220-019-00174-7

\bibitem[Xu2024]{Xu2024}Xu, G.: Integral of scalar curvature on manifolds
with a pole. Proceedings AMS, 152 (2024), no. 11, 4865\textendash 4872

\bibitem[Ya2013]{Yang2013}Yang, B.: On a problem of Yau regarding
a higher dimensional generalization of the Cohn\textendash Vossen
inequality. Math. Ann. 355, 765\textendash 781 (2013). https://doi.org/10.1007

\bibitem[Yau1992]{Yau1992}Yau, S.T.: Open problems in geometry. Chern\textendash a
great geometer of the twentieth century. International Press, Hong
Kong, 1992, 275-319

\bibitem[Zh2022]{Zh2022}Zhu, B.: Geometry of positive scalar curvature
on complete manifold. J. Reine Angew. Math. 791 (2022), 225\textendash 246

\bibitem[Zh1994]{Zhu1994}Zhu, S.: The classification of complete
locally conformally flat manifolds of nonnegative Ricci curvature.
Pac. J. Math. 163(1), 189\textendash 199 (1994)

\end{thebibliography}
\end{document}